\documentclass[11pt,amssymb,amstex]{amsart}
\usepackage{amssymb}
\usepackage{amsmath}
\usepackage{amsthm}
\usepackage{amsfonts}
\usepackage{graphicx}
\usepackage{epstopdf}

\newtheorem{definition}{Definition}[section]
\newtheorem{theorem}{Theorem}[section]
\newtheorem{corollary}[theorem]{Corollary}
\newtheorem{lemma}[theorem]{Lemma}
\newtheorem{remark}[theorem]{Remark}
\newtheorem{example}[theorem]{Example}
\newtheorem{proposition}[theorem]{Proposition}
\newtheorem{conjecture}[theorem]{Conjecture}
\newtheorem{problem}[theorem]{Problem}

\numberwithin{equation}{section}

\def\ca{\mathcal{ A}}

\def\ce{\mathcal{ E}}

\def\gl{{\frak L}}

\def\bb{{\mathbb B}}

\def\bd{{\mathbb D}}

\def\bk{{\mathbb K}}

\def\bn{{\mathbb N}}

\def\br{{\mathbb R}}

\def\a{\alpha}
\def\b{\beta}
\def\g{\gamma}  
\def\d{\delta}  \def\D{\Delta}

\def\e{\epsilon}
\def\l{\lambda} 

\def\m{\mu}

 \def\O{\Omega}

\begin{document}

\title[Elliptic Quadratic Operator Equations]{Elliptic Quadratic Operator Equations}

\author{Rasul Ganikhodjaev}
\address{Rasul Ganikhodjaev\\
Faculty of Mechanics \& Mathematics, National University Uzbekistan\\
Tashkent\\ Uzbekistan} \email{{\tt rganikhodzhaev@gmail.com}}

\author{Farrukh Mukhamedov}
\address{Farrukh Mukhamedov\\
Department of Mathematical Sciences\\
College of Science, The United Arab Emirates University\\
P.O. Box, 15551, Al Ain\\
Abu Dhabi, UAE} \email{{\tt far75m@gmail.com} {\tt
farrukh.m@uaeu.ac.ae}}

\author{Mansoor Saburov}
\address{Mansoor Saburov\\
Department of Computational \& Theoretical Sciences \\
Faculty of Science, International Islamic University Malaysia\\
P.O. Box, 25200, Kuantan\\
Pahang, Malaysia} \email{{\tt msaburov@gmail.com}}

\begin{abstract}
In the present paper is devoted to the study of elliptic quadratic
operator equations over the finite dimensional Euclidean space. We
provide necessary and sufficient conditions for the existence of
solutions of elliptic quadratic operator equations. The iterative
Newton-Kantorovich method is also presented for stable solutions.


\end{abstract}

\maketitle


\tableofcontents


\section{Introduction}


\subsection{Hammerstein integral equations}

A nonlinear Hammerstein integral equation appeared
\begin{equation}\label{integequat}
x(t)=\int\limits_{\O}\int\limits_{\O} K_1(t,s,u)x(s)x(u)dsdu+\int\limits_{\O}K_2(t,s)x(s)ds+f(t)
\end{equation}
in several problems of astrophysics, mechanics, and biology, where  $K_1:\O\times\O\times\O\to\br,$ $K_2:\O\times\O\to\br,$ and $f:\O\to\br$ are given functions and $x:\O\to\br$ is an unknown function. Generally, in order to solve the nonlinear Hammerstein integral equation \eqref{integequat} over some functions space,  one should impose some constrains for the functions $K_1(\cdot,\cdot,\cdot), K_2(\cdot,\cdot),$ and $f(\cdot)$. For instance, by using contraction methods, some sufficient conditions were obtained for the existence of solutions of the integral equation \eqref{integequat} over the space $C(\O)$ of continuous functions (see \cite{Atkinson}, \cite{Golberg1979,Golberg1990}, \cite{Krasnoselskii}, \cite{Some}). It is worth of noting that, unlike a linear integral equation (i.e. $K_1(t,s,u)\equiv0$), in general, the nonlinear Hammerstein integral equation \eqref{integequat} may have many solutions.

Particularly, if $K_1$ and $K_2$ are Goursat's degenerate kernels, i.e.
\begin{eqnarray}
K_1(t,s,u)=\sum\limits_{i,j,k=1}^na_i(s)b_j(u)c_k(t),\\
K_2(t,s)=\sum\limits_{i,j=1}^nd_i(s)e_j(t),
\end{eqnarray}
where $a_i(\cdot), \ b_i(\cdot), \ c_i(\cdot), \ d_i(\cdot), \ e_i(\cdot)$ are given functions then we have that
\begin{equation*}
x(t)=\sum\limits_{i,j,k=1}^n\left(\int\limits_\O a_i(s)x(s)ds\right)\left(\int\limits_\O b_j(u)x(u)du\right) c_k(t)+\sum\limits_{i,j=1}^n\left(\int\limits_\O d_i(s)x(s)ds\right)e_j(t)+f(t).
\end{equation*}

Let
\begin{eqnarray*}
\int\limits_\O a_i(s)x(s)ds=x_i,\quad \int\limits_\O b_j(s)x(s)ds=x_{n+j},\quad \int\limits_\O d_k(s)x(s)ds=x_{2n+k},\ \ i,j,k=\overline{1,n}.
\end{eqnarray*}
In this setting, the solution of the nonlinear Hammerstein integral equation \eqref{integequat} takes the following form
$$x(t)=\sum_{i,j,k=1}^nx_ix_{n+j}c_k(t)+\sum\limits_{i,j=1}^nx_{2n+i}e_{j}(t)+f(t),$$
where $x=(x_1,\cdots,x_n,x_{n+1},\cdots,x_{2n},x_{2n+1},\cdots,x_{3n})\in\br^{3n}$ is a solution of the following quadratic operator equation
\begin{eqnarray}\label{QOE}
\sum\limits_{i,j=1}^{3n}A_{ij,k}x_ix_j+\sum\limits_{i=1}^{3n}B_{ik}x_{i}+C_{k}=0,\quad \quad \forall \ k=\overline{1,3n}.
\end{eqnarray}
for suitable $(A_{ij,k})_{i,j,k=1}^{3n},$ $(B_{ik})_{i,k=1}^{3n},$ and $(C_k)_{k=1}^{3n}$.

Consequently, in order to find solutions of the nonlinear Hammerstein integral equation \eqref{integequat} with Goursat's degenerate kernels, we have to solve the quadratic operator equation \eqref{QOE} over $\mathbb{R}^{3m}$.

Let $Q:\br^n\to\br^m$ be a quadratic operator
\begin{equation*}
Q(x)=\bigg(\sum_{i,j=1}^na_{ij,1}x_ix_j,\sum_{i,j=1}^na_{ij,2}x_ix_j,\cdots,\sum_{i,j=1}^na_{ij,m}x_ix_j\bigg),
\end{equation*}
where $a_{ij,k}\in \br$ are structural coefficients and $x=(x_1,\cdots,x_n)\in \br^n$. Without loss of generality, one can assume that $a_{ij,k}=a_{ji,k}$ for any $i,j=\overline{1,n}$ and $k=\overline{1,m}$. Let $A_k=(a_{ij,k})_{i,j=1}^n$ be a symmetric matrix for $k=\overline{1,m}$. In this case, the quadratic operator can be written in the following form
\begin{equation*}
Q(x)=((A_1x,x),(A_2x,x),\cdots,(A_mx,x))
\end{equation*}

Let $H_{n,m}(Q)$ be a real linear span of symmetric matrices $A_1,\cdots, A_m$.
We say that $H_{n,m}(Q)$ is \textit{positive definite} (resp. \textit{positive semidefinite}) if there exists a positive definite (resp. a nonzero positive semidefinite but not positive definite) matrix in it. We say that $H_{n,m}(Q)$ is \textit{indefinite} if every nonzero matrix in it is indefinite. Let $R_{n,m}(Q)=\{Q(x): x\in\br^n\}$ and $W_{n,m}(Q)=\{Q(x): \|x\|_2=1\}$ be the images of $\br^n$ and $\mathbb{B}(0)=\{x\in\br^n: \|x\|_2=1\}$, respectively, under the quadratic operator. Let $Ker_{n,m}(Q)=\{x\in\br^n: Q(x)=0\}$ be a kernel of the quadratic operator.

The study of convexity of the sets  $R_{n,m}(Q)$, $W_{n,m}(Q)$ and on the relationship between the sets $H_{n,m}(Q)$ and $Ker_{n,m}(Q)$ are traced back to O. Toeplitz \cite{Toeplitz}, F. Hausdorff \cite{Hausdorff}, P. Halmos \cite{Halmos}, C.A. Berger \cite{Berger}, R. Westwick \cite{Westwick}, P. Finsler \cite{Finsler1936a,Finsler1936b}, G.A. Bliss \cite{Uhlig}, W.T. Reid \cite{Reid}, A.A. Albert \cite{Albert}, E.J. McShane \cite{McShane}, M. Hestenes \cite{HestenesMcShane,Hestenes1951,Hestenes1968}, F. John \cite{John}, L. Dines \cite{Dines1941,Dines1942,Dines1943}, and many others (see also \cite{Au-Yeung1969a}-\cite{Au-Yeung1975b}).

Let us consider the quadratic operator equation
\begin{eqnarray}\label{GeneralQOE}
Q(x)+Ax+b=0, \quad x\in\br^n
\end{eqnarray}
where $Q:\br^n\to\br^m$ is a quadratic operator and $A:\br^n\to\br^m$ is a linear operator, and $b\in\br^m$ is a vector.

\subsection{Summary of the main results}
The main goal is to study the structure of the set
$$X_{n,m}(Q,A,b)=\{x\in\br^n: Q(x)+Ax+b=0\}$$ of solutions of quadratic operator equation.

\begin{definition} A quadratic operator $Q:\br^n\to \br^m$ is called
\begin{enumerate}
\item[(i)] {\it elliptic} (in short EQO) if there exists a linear continuous functional $f:\br^m\to\br$
such that  $f(Q(x))$ is a positive definite quadratic form;
\item[(ii)] {\it parabolic} (in short PQO) if there exists
a nonzero  linear continuous functional $f:\br^m\to\br$ such that
$f(Q(x))$ is a positive semidefinite but no positive definite quadratic form;
\item[(iii)] {\it hyperbolic} (in short HQO) if for any nonzero  linear continuous functional $f:\br^m\to\br$ the quadratic form $f(Q(x))$ is indefinite.
\end{enumerate}
\end{definition}

\begin{remark}
It is clear that the quadratic operator $Q:\br^n\to \br^m$ is elliptic (resp. parabolic, hyperbolic) if and only if $H_{n,m}(Q)$ is positive definite (resp. positive semidefinite, indefinite).
\end{remark}

Consequently, by means of Dines's result, we can fully describe all elliptic, parabolic, hyperbolic quadratic operators. Namely, we have the following result.

\begin{theorem}\cite{Dines1941,Dines1942,Dines1943}
Let $Q:\br^n\to \br^m$ be a quadratic operator. The following statements hold true:
\begin{itemize}
\item[(i)] $Q$ is elliptic if and only if $tr(SA_i)=0$ with $S=S^{T}$ for all $i=\overline{1,m}$ implies that $S$ is indefinite.

\item[(ii)] $Q$ is parabolic if and only if there exists positive semidefinite $S=S^T$ with $tr(SA_i)=0$ for all $i=\overline{1,m}$, but no such positive definite $S$.

\item[(iii)] $Q$ is hyperbolic if and only if there exists positive definite $S=S^T$ with $tr(SA_i)=0$ for all $i=\overline{1,m}$.
\end{itemize}
\end{theorem}

There is a strong relation between the convexity of the sets $R_{n,m}(Q)$, $W_{n,m}(Q)$ and the uniqueness of the set $Ker_{n,m}(Q)$ whenever $Q:\br^n\to \br^m$ is the elliptic quadratic operator and $n\geq m$ (see \cite{Afriat,Agrachev}, \cite{Arutyunov1991}-\cite{ArutyunovRozova}, \cite{BaccariSamet}-\cite{Barvinok2014}, \cite{Hiriart-UrutyTorki2002,Hiriart-Uruty2007,Hiriart-Uruty2009}, \cite{Matveev1996}-\cite{Polyak2003}, \cite{Sheriff},\cite{Vershik},\cite{Yakubovich}). In this survey paper, we are going to describe the set $X_{n,m}(Q,L,b)$ whenever $Q:\br^n\to \br^m$ is the elliptic quadratic operator and $n=m$.

Let $Q:\br^n\to \br^n$ be the elliptic quadratic operator.
We know that, the quadratic form $f(Q(x))$ is
positive definite in $\br^n$ if and only if there exists a positive number $\alpha >0$ such that
\begin{equation*}
f(Q(x))\geq \alpha\cdot \|x\|_2^2,
\end{equation*}
for any $x\in \br^n$. Thus, $Q:\br^n\to \br^n$ is the elliptic quadratic operator if and only if there exist a continuous linear functional $f:\br^n\to\br$ and a number $\alpha >0$ such that
\begin{equation*}
f(Q(x))\geq \alpha\cdot \|x\|_2^2, \quad \forall \ x\in \br^n.
\end{equation*}
Let $K'_Q$ be a set of all continuous linear functionals $f:\br^n\to\br$ such that the quadratic form $f(Q(x))$ is positive defined, i.e.,
\begin{equation*}
K'_Q=\left\{f: \ \exists \ \alpha_f>0 ,\  f(Q(x))\geq \alpha_f\|x\|_2^2, \ \forall \ x\in \br^n\right\}.
\end{equation*}

It is clear that $K'_Q\ne \emptyset$.

\begin{proposition}
If $Q:\br^n\to \br^n$  is the elliptic quadratic operator then $K'_Q$ is an open convex cone. Moreover, for any given minihedral cone $K\subset \br^{n}$
there exists the elliptic quadratic operator $Q:\br^n\to \br^n$ such that ${\overline {K'_Q}}=K.$
\end{proposition}

For every $f\in K'_Q$, we define an {\it ellipsoid}
\begin{equation*}
\ce_f=\{x\in\br^n: f(Q(x)+Ax+b)\leq 0\},
\end{equation*}
corresponding to $f\in K'_Q.$ We define the following set  $$\ce_n(Q,A,b)\equiv\bigcap\limits_{f\in K'_Q}\ce_f.$$

\begin{theorem}
If the equation \eqref{GeneralQOE} is solvable then
$\ce_n(Q,A,b)\neq \emptyset$.
\end{theorem}


The following theorem gives more accurate description of the set $X_n(Q,A,b).$

\begin{theorem}
If the equation \eqref{GeneralQOE} is solvable  then $X_n(Q,A,b) \subset {\textbf{\textup{Extr}}} (\ce_n(Q,A,b))$.
\end{theorem}

The following theorem gives a solvability criterion for the elliptic operator equation \eqref{GeneralQOE}.


\begin{theorem}
The elliptic operator equation \eqref{GeneralQOE} is solvable if and only if $\bigcap\limits_{f\in K'_Q}\partial \ce_f\neq \emptyset.$  Moreover, if the elliptic operator equation \eqref{GeneralQOE} is solvable  then $X_n(Q,A,b)= \bigcap\limits_{f\in K'_Q}\partial \ce_f$.
\end{theorem}


Let $\bk$ be the set of extreme rays of the closed cone
${\overline K}'_Q.$ We define a set
\begin{equation*}
\Pi_f(\bk)=\{x\in\br^n: f(Q(x)+Ax+b)\leq 0, \ \ f\in \bk\}.
\end{equation*}

\begin{proposition} One has that $\bigcap\limits_{f\in\bk}\Pi_f(\bk)=\ce_n(Q,A,b).$
\end{proposition}

\begin{theorem}
Every vertex  of  $\ce_n(Q,A,b)$ is a solution of the elliptic operator equation \eqref{GeneralQOE}.
\end{theorem}

A solution of the elliptic operator equation \eqref{GeneralQOE} which is the vertex of $\ce_n(Q,L,b)$ has a special property among other solutions. We know that the set of all elliptic quadratic operators are closed under the small perturbation.

\begin{definition}
A solution $x_0$ of the elliptic operator equation \eqref{GeneralQOE}
is called {\it stable} if for any $\e>0$ there exists $\d>0$ such that the
perturbed elliptic operator equation ${\tilde Q}(x)+{\tilde A}x+{\tilde b}=0$
has a solution ${\tilde x}_0$ such that
$\|{\tilde x}_0-x\|<\e$ whenever $\|{\tilde Q}-Q\|<\d, \ \|{\tilde A}-A\|<\d,$ $\|{\tilde b}-b\|<\d.$
\end{definition}

\begin{theorem}
A solution of the elliptic operator equation \eqref{GeneralQOE} is stable if and only if it is a vertex of $\ce_n(Q,L,b)$.
\end{theorem}

We could speak more about the stable solutions of the elliptic operator equation \eqref{GeneralQOE}.

\begin{theorem}
An elliptic operator equation \eqref{GeneralQOE} has an even (possibly, zero) number of stable solutions.
\end{theorem}

We can also approximate the stable solutions of elliptic operator equation \eqref{GeneralQOE} by the Newton-Kantorovich method. It is easy to check that the set
\begin{equation*}
\bd=\br^n\setminus \bigg(\bigcup_{f\in \bk}\Pi_f(\bk)\bigg)
\end{equation*}
is an open set. Let $\bd_0$ be a connected component of $\bd$ and ${\overline
\bd}_0$ be its closure.

Let $P:\br^n\to\br^n$ be a mapping defined as $P(x):=Q(x)+Ax+b$ for any $x\in\br^n$.

\begin{theorem} If there exists $x_0\in \bd_0$ such that ${\overline \bd}_0$ does
not contain any straight line passing through $x_0$ then there
exists a stable solution $x_{*}$ of the elliptic operator equation \eqref{GeneralQOE} which belongs to ${\overline \bd}_0.$ Moreover, the inverse $[P'(x_0)]^{-1}$ of the mapping $P'(x_0)$ exists and the sequence $\{x_k\}_{k=1}^\infty$ defined as follows
\begin{equation*}
x_{k+1}=x_k-[P'(x_k)]^{-1}P(x_k), \ \ k=0,1,\dots
\end{equation*}
converges to the stable solution $x_{*}$.
\end{theorem}

We are aiming to classify the set of elliptic operators based on their ranks.

Let $Q:\br^n\to\br^n$ be an elliptic quadratic operator.  Let $\bk$
be the set of extremal rays of ${\overline K}'_Q.$ Then, due to
Krein-Milman theorem, we have that
$\overline{conv(\bk)}={\overline K}'_Q,$ where $conv(\bk)$ is a convex
hull of $\bk.$ Let ${\rm rg}_fQ$ stand for the rank of the quadratic
form $f(Q(x))$. It is clear that the rank ${\rm rg}_fQ$ of the quadratic form  $f(Q(x))$ is equal to the rank of the associated symmetric matrix $A$. Due to the construction of the set $ K'_Q$, one has that ${\rm rg}_fQ=n$ whenever $f\in K'_Q$ and ${\rm rg}_fQ<n$ whenever $f\in
\partial K'_Q$.

\begin{definition} The
number
$${\rm rg} Q=\max_{f\in \bk}{\rm rg}_fQ$$ is called a {\it rank} of the elliptic quadratic operator $Q:\br^n\to\br^n$.
\end{definition}

It is clear  that $1\leq {\rm rg} Q\leq n-1$ for any elliptic quadratic operator. Moreover, if $A,B$ are invertible matrices such that $AQ(B(\cdot))$ also is an elliptic quadratic operator then ${\rm rg}(AQ(B))={\rm rg Q}.$

\begin{definition} An elliptic quadratic operator
$Q:\br^n\to\br^n$ is called {\it homogeneous}
of rank $k$, if one has that ${\rm rg}_fQ=k$ for any $f\in \bk$.
\end{definition}

We can describe the cone ${\overline K}'_Q$ for a homogeneous elliptic quadratic operators of order $k$.

\begin{theorem} Let $Q$ be a homogeneous elliptic quadratic operator. One has that ${\rm rg}Q=1$ if and
only if ${\overline K}'_Q$ is a miniedral cone, i.e,  $\bk$ contains
exactly $n$ extremal rays. Moreover, if ${\rm rg}Q=1$ then there exist invertible matrices $A, B$ such that $AQ(Bx)=(x_1^2,x^2_2,\dots,x_n^2)$.
 \end{theorem}

\begin{theorem} Let $Q$ be a homogeneous elliptic quadratic operator. If ${\rm rg}Q\geq 2$ then $\bk$ is an infinite set. Moreover, if ${\rm rg}Q=n-1$
then $\bk=\partial {\overline K}'_Q$.
\end{theorem}

%
%
%

In general, it is a tedious work to describe the cone $\overline{K}'_Q$ of homogeneous elliptic quadratic operators with rank $2\leq{\rm rg}Q\leq n-2$.  It can be observed in some examples.

%
%

We can provide some explicit sufficient conditions for the solvability of elliptic rank-1 operator equation \eqref{GeneralQOE}. Let $Q:\br^n\to \br^n$ be an elliptic operator of the rank-1. Then, the elliptic operator equation \eqref{GeneralQOE} can be written as follows

\begin{equation}\label{rank1}
x_k^2=\sum^n_{i=1}a_{ki}x_i + b_k; \quad k=\overline{1,n}.
\end{equation}

\begin{theorem}\label{twostablesolution1}
Let $A=(a_{ij})_{i,j=1}^n$ be a matrix such that
$a_{i_1j}\cdot a_{i_2j}\geq 0$ for all $i_1,i_2,j=\overline{1,n}.$
If one has that
\begin{equation}\label{condforrank1}
\bigg(\sum\limits_{j=1}^n\min\limits_{i=\overline{1,n}}|a_{ij}|\bigg)^2+4 \min\limits_{i=\overline{1,n}} b_i> 0
\end{equation}
then the elliptic operator equation \eqref{rank1} has at least two stable solutions.
\end{theorem}

\begin{remark} In the case $n=1$, the condition \eqref{condforrank1} coincides
with the positivity of the discriminant of the quadratic equation $x^2=ax+b$. Hence, the condition \eqref{condforrank1} is a necessary and sufficient condition for the existence of two stable
solutions whenever $n=1$.
\end{remark}

%
%
%
%
%
%

\section{A classification of quadratic operators}


In this section we are going to classify quadratic operators into three classes and to study their properties. In what follows, we shall consider quadratic operators on the finite dimensional Euclidian space $\br^n$.

Let $B:\br^n\times \br^n\to \br^n$ be  a symmetric
bilinear operator. A quadratic operator $Q:\br^n\to \br^n$ is defined as follows
\begin{equation*}
    Q(x)=B(x,x),\quad \forall \ x\in \br^n.
\end{equation*}
It is well-known that every quadratic operator $Q:\br^n\to \br^n$  uniquely defines the symmetrical
bilinear operator $B:\br^n\times \br^n\to \br^n$ associated with the given quadratic operator $Q$
\begin{equation*}
    B(x,y)=\frac{1}{4}[Q(x+y)-Q(x-y)]=\frac{1}{2}[Q(x+y)-Q(x)-Q(y)].
\end{equation*}
Moreover, every quadratic operator $Q:\br^n\to\br^n$ can be
written in the coordinate form as follows
\begin{equation*}
    Q(x)=\bigg(\sum_{i,j=1}^na_{ij,1}x_ix_j,\sum_{i,j=1}^na_{ij,2}x_ix_j,\cdots,\sum_{i,j=1}^na_{ij,n}x_ix_j\bigg),
\end{equation*}
where $a_{ij,k}\in \br$ are structural coefficients and $x=(x_1,\dots,x_n)\in \br^n$. Without loss any generality, one can assume that $a_{ij,k}=a_{ji,k}.$ We denote the set of all quadratic operators acting on $\br^n$ by ${\mathfrak{Q}}_n.$

Any quadratic operator $Q:\br^n\to \br^n$ is {\it bounded}, i.e., there exists a positive number $M>0$ such that
\begin{equation*}
    \|Q(x)\|\leq M\cdot \|x\|^2,  \ \forall  \ x\in \br^n,
\end{equation*}
and it is {\it continuous.} Let us define the norm of the quadratic operator $Q$ by
$$\|Q\|=\sup_{\|x\|\leq
    1}\|Q(x)\|. $$
It is clear that $\|Q(x)\|\leq \|Q\|\cdot\|x\|^2$ for any $x\in\br^n.$

One can see that the set ${\mathfrak{Q}}_n$ forms the $\frac{n^2(n+1)}{2}-$ dimensional normed space
with the quadratic operator norm. We are going to classify quadratic operators into three classes.

\begin{definition} A quadratic operator $Q:\br^n\to \br^m$ is called
    \begin{enumerate}
        \item[(i)] {\it elliptic} (in short EQO) if there exists a linear continuous functional $f:\br^m\to\br$
        such that  $f(Q(x))$ is a positive definite quadratic form;
        \item[(ii)] {\it parabolic} (in short PQO) if there exists
        a nonzero  linear continuous functional $f:\br^m\to\br$ such that
        $f(Q(x))$ is a positive semidefinite but no positive definite quadratic form;
        \item[(iii)] {\it hyperbolic} (in short HQO) if for any nonzero  linear continuous functional $f:\br^m\to\br$ the quadratic form $f(Q(x))$ is indefinite.
    \end{enumerate}
\end{definition}

We denote the sets of all elliptic quadratic, parabolic quadratic, and hyperbolic quadratic operators acting on $\br^n$ by ${\mathfrak{EQ}}_n,$ ${\mathfrak{PQ}}_n,$ and ${\mathfrak{HQ}}_n,$ respectively.

We know that, the quadratic form $f(Q(x))$ is
positive defined in a finite dimensional vector space if and only if there exists a positive number $\alpha >0$ such that
\begin{equation*}
    f(Q(x))\geq \alpha\cdot \|x\|^2,
\end{equation*}
for any $x\in \br^n$. Thus $Q:\br^n\to \br^n$ is an EQO if and only if there exist a continuous linear
functional $f:\br^n\to\br$ and a number $\alpha >0$ such that
\begin{equation*}
    f(Q(x))\geq \alpha\cdot \|x\|^2, \quad \forall \ x\in \br^n.
\end{equation*}

First of all, we shall study some basic properties of quadratic operators.

Let $Q:\br^n\to \br^n$  be an EQO and  $K'_Q$ be a set of all continuous linear functionals $f:\br^n\to\br$ such that the quadratic form $f(Q(x))$ is positive defined, i.e.,
\begin{equation}\label{K'Q}
    K'_Q=\left\{f: \ \exists \ \alpha_f>0 ,\  f(Q(x))\geq \alpha_f\|x\|^2, \ \forall \ x\in \br^n\right\}.
\end{equation}
Due to the definition of EQO we have that $K'_Q\ne \emptyset$.

We recall that a set $K\subset\br^n$ is called a cone if $\lambda K\subset K$ for any $\lambda >0$ and $K\cap (-K)=\emptyset.$

\begin{proposition}
    If $Q:\br^n\to \br^n$  is an EQO and $K'_Q$ is defined by \eqref{K'Q} then $K'_Q$ is an open convex cone.
\end{proposition}

\begin{proof} Let us prove that $K'_Q$ is an open set. If $f_0\in K'_Q$ then  there exists $\a_{0}>0$ such that $f_0(Q(x))\geq \alpha_{0}\|x\|^2$ for any $x\in \br^n$. Since $Q$ is bounded, we have
    \begin{equation*}
        \|Q(x)\|\leq M\cdot \|x\|^2,
    \end{equation*}
    for some $M>0$ and for any $x\in \br^n.$ If we take $\varepsilon=\frac{\a_{0}}{2M}$ then for any linear functional $f:\br^n\to\br$ with $\|f-f_0\|<\varepsilon$ we get
    \begin{eqnarray*}
        f(Q(x))&=&f_0(Q(x))+(f-f_0)(Q(x))\geq \a_{0}\|x\|^2-
        \|f-f_0\|\cdot \|Q(x)\|\\
        &\geq&  \a_{0}\|x\|^2-\frac{\a_{0}}{2M}\cdot M\cdot \|x\|^2=\frac{\a_{0}}{2}\|x\|^2.
    \end{eqnarray*}
    This means that $f\in K_Q',$ and $K'_Q$ is an open set.

    Let us show that $K'_Q$ is a convex set. If  $f_1,f_2\in K'_Q$ then one has $f_i(Q(x))\geq \a_i
    \|x\|^2 $ for any $x\in \br^n$ where $\a_i>0, \ i=1,2.$ Let $f_\l=\l f_1+(1-\l) f_2, \ 0\leq
    \l \leq 1$. We then get
    \begin{eqnarray*}
        f_\l(Q(x))&=&\l f_1(Q(x))+(1-\l)f_2(Q(x))\geq(\l \a_1+(1-\l) \a_2)\cdot \|x\|^2 \geq \min\{\a_1, \a_2\}\cdot
        \|x\|^2.
    \end{eqnarray*}
    Hence, $ f_\l\in K'_Q$  for any $0\leq
    \l \leq 1,$ i.e., $K'_Q$ is a convex set.

    It immediately follows from the definition of the set $K'_Q$ that $\lambda K'_Q\subset K'_Q$ for any $\lambda >0$ and $K'_Q\cap (-K'_Q)=\emptyset.$ This means that $K'_Q$ is a cone.
\end{proof}

\begin{remark}
    It is worth mentioning that a closure ${\overline K}'_Q$ of $K'_Q$  may  not be a cone. For
    example, let us consider the following EQO on $\br^2$
    $$Q(x)=(x^2_1+x^2_2, x^2_1+x^2_2).$$
    Then $K'_Q=\{(\a,\b): \a+\b>0\}.$ However, the set  ${\overline
        K}'_Q=\{(\a,\b): \a+\b\geq 0\}$ is a semi-plane which is not a cone.
\end{remark}

Recall, given a cone $K\subset \br^{n}$, we can define a partial ordering $\leq_K$ with respect to $K$ by $x\leq_{K} y$ if $y-x\in K.$ The cone $K$ is called {\it minihedral} if $\sup\{x,y\}$ exists for any $x,y\in  \br^{n}$, where the supremum is taken with respect to the partial ordering $\leq_{K}.$
It is well-known that $K\subset\br^n$ is a minihedral cone if and only if it is a conical hull of $n$
linear independent vectors, i.e.,
\begin{equation*}
    K=cone\{z_1,\dots,z_n\}=\bigg\{x: x=\sum^n_{i=1}\l_i z_i, \
    \l_i\geq 0\bigg\}.
\end{equation*}

\begin{proposition}\label{11ell2} For any given minihedral cone $K\subset \br^{n^{*}}$
    there exists an EQO such that ${\overline K}'_Q=K.$
\end{proposition}

\begin{proof}
    Let
    $$C=\{x\in\br^n: f(x)\geq 0, \
    \forall f\in K\}$$
    be a dual cone to the given minihedral cone $K.$ It is known that $C$  is also a minihedral cone. Without loss of
    generality, we may suppose that
    $$C=cone\{e_1,\cdots,e_n\},$$
    where $e_i=(\d_{1i}, \dots ,\d_{ni}) $ and
    \begin{equation*}
        \d_{ij}=\left\{\begin{array}{ll}
            1 \ \mbox{if}\ \ i=j\\
            0 \ \mbox{if} \ \ i\ne j.
        \end{array}\right.
    \end{equation*}

    We define the quadratic operator $Q:\br^n\to\br^n$ as follows
    $$Q(x)=(x^2_1,\dots,x^2_n),$$
    where $x=\sum\limits^n_{i=1}x_ie_i.$

    Let $f=(\l_1,\dots,\l_2)$ be a linear functional. Then a quadratic
    form
    \begin{equation*}
        f(Q(x))=\sum^n_{i=1}\l_i x_i^2,
    \end{equation*}
    is positive defined if and only if  $\l_1>0,\dots,\l_n>0.$ Consequently,
    \begin{equation*}
        K'_Q=\{(\l_1,\dots,\l_n): \l_1>0,\dots,\l_n>0\}
    \end{equation*} and
    \begin{equation*}
        {\overline K}'_Q=\{(\l_1,\dots,\l_n): \l_1\geq 0,\dots,\l_n\geq
        0\}=cone\{f_1,\dots,f_n\},
    \end{equation*}
    where $f_i=(\d_{1i},\dots,\d_{ni}).$ Since $f_i(e_j)=\d_{ij}$ we hence have ${\overline K}'_Q=K.$
\end{proof}

\begin{lemma}\label{l0} If $K_1$ and $K_2$ are open cones in $\br^n,$ $n\geq 2$
    then there exists a minihedral cone $K$ such that
    \begin{equation*}
        K_1\cap intK\ne \emptyset, \quad K_2\cap intK\ne \emptyset,
    \end{equation*}
    where $intK$ is an interior of $K.$
\end{lemma}
\begin{proof}
    Let $K_1$ and $K_2$ be open cones. Then we can take  $y_i\in K_i,$ $i=1,2$ such that $y_1$ and $y_2$ are
    linear independent. We complete these vectors $\{y_1,y_2\}$ up to a base $\{y_1,y_2,\cdots,y_n\}$ of $\br^n.$
    Then, it is easy to see that the minihedral cone $K=cone\{y_1,y_2,\cdots,y_n\}$ satisfies
    all conditions of the lemma.
\end{proof}

\begin{proposition}\label{p3} The set ${\mathfrak{EQ}}_n$ is a path connected subset of ${\mathfrak{Q}}_n$ whenever $n\geq 2$.
\end{proposition}
\begin{proof}
    Let $Q_i,$ $i=0,1$ be elliptic operators and $K'_{Q_i},$ $i=0,1$ be
    the corresponding open cones. Due to Lemma \ref{l0} we can choose
    the minihedral cone $K$ such that $K'_{Q_i}\cap intK\ne \emptyset, \
    i=0,1.$ According to Proposition \ref{11ell2} we can construct an
    elliptic operator $Q$ such that ${\overline K}'_Q=K$. We define a quadratic operator $Q_\l:\br^n\to\br^n$ as follows
    \begin{equation*}
        Q_\l=\left\{\begin{array}{ll}
            2\l Q+(1-2\l)Q_0 \ \  \ \ \ \textrm{if}\ \ 0\leq \l\leq \frac{1}{2}\\
            2(1-\l) Q+(2\l-1)Q_1 \ \textrm{if}\ \ \frac{1}{2}< \l\leq 1.\\
        \end{array}\right.
    \end{equation*}

    Let us show that $Q_\l$ is elliptic for any $\l\in [0,1]$. If
    $f_0\in K'_{Q_0}\cap intK$ and $0\leq \l \leq
    \frac{1}{2}$ then one has $f_0(Q_0(x))\geq \a_0 \|x\|^2,\ \a_0>0$
    and $f_0(Q(x))\geq \b_0 \|x\|^2, \ \b_0>0.$ Hence
    \begin{eqnarray*}
        f_0(Q_\l(x))&=&f_0(2\l Q(x)+(1-2\l)Q_0(x))\geq
        (2\l\a_0+(1-2\l)\b_0)\cdot \|x\|^2\geq \min \{\a_0, \b_0\}\cdot \|x\|^2.
    \end{eqnarray*}

    Similarly for $f_1\in K'_{Q_1}\cap intK$ and
    $\frac{1}{2}< \l \leq 1$ we get that $f_1(Q_\l(x))\geq \min \{\a_1,\b_1\}\cdot \|x\|^2$ where $\a_1>0, \b_1>0$ such that $f_1(Q_1(x))\geq \a_1 \|x\|^2,$ $f_1(Q(x))\geq \b_1 \|x\|^2.$ This completes the proof.
\end{proof}

\begin{remark} It is worth noting that in the case $n=1,$ Lemma \ref{l0} and Proposition \ref{p3}
    are not true. Indeed, in this case any quadratic operator has a
    form $Q(x)=ax^2,$ and the elipticity means that $a\ne 0.$ Thus,
    the set of all elliptical operators in one dimensional setting is
    $\br^1\setminus \{0\}$, which is not connected.
\end{remark}

\begin{proposition}\label{p4} The set ${\mathfrak{EQ}}_n$ is an open subset of ${\mathfrak{Q}}_n.$
\end{proposition}

\begin{proof} Let  $Q_0:\br^n\to\br^n$ be EQO, then there is a linear functional $f_0:\br^n\to\br$ such that
    $f_0(Q_0(x))\geq \a_0 \|x\|^2$ for some $\a_0>0$.  We then want to show that
    $$\{Q:\|Q-Q_0\|<\frac{\a_0}{2\|f_0\|}\}\subset {\mathfrak{EQ}}_n.$$
    Indeed, if $Q:\br^n\to\br^n$ is a quadratic operator with
    \begin{equation*}
        \|Q(x)-Q_0(x)\|< \frac{\a_0}{2\|f_0\|}\cdot \|x\|^2, \quad \forall \ x\in\br^2,
    \end{equation*}
    then one has that
    \begin{eqnarray*}
        f_0(Q(x))&=& f_0(Q_0(x))+f_0(Q(x)-Q_0(x))\geq  \a_0 \|x\|^2-\|f_0\|\cdot
        \frac{\a_0}{2\|f_0\|}\cdot \|x\|^2= \frac{\a_0}{2}\cdot \|x\|^2.
    \end{eqnarray*}
    This means that  $Q$ is the EQO and ${\mathfrak{EQ}}_n$ is the open subset of ${\mathfrak{Q}}_n.$
\end{proof}

Analogously, one can prove the following statement.

\begin{proposition}\label{p5} The following statements hold true:
    \begin{itemize}
        \item [(i)] The set ${\mathfrak{PQ}}_n$ is a closed and path connected subset of
        ${\mathfrak{Q}}_n$ with empty interior;
        \item [(ii)] The set ${\mathfrak{HQ}}_n$ is an open subset of ${\mathfrak{Q}}_n.$
    \end{itemize}
\end{proposition}

\section{Examples}

We are going to provide some examples for quadratic operators.

\begin{example}[The classification of quadratic operators on $\br^2$]

    Let us consider the quadratic operator acting on $\br^2,$ i.e.,
    \begin{equation*}
    Q(x)=(a_1x^2_1+2b_1x_1x_2+c_1x^2_2, a_2x_1^2+2b_2x_1x_2+c_2x_2^2),
    \end{equation*}
    where $x=(x_1,x_2)\in \br^2.$ We denote by
    \begin{equation*}
    \Delta = \left|
    \begin{array}{cc}
    a_1 & c_1\\[2mm]
    a_2 & c_2\\
    \end{array}
    \right|^2- 4\left|
    \begin{array}{cc}
    a_1 & b_1\\[2mm]
    a_2 & b_2\\
    \end{array}
    \right|\cdot \left|
    \begin{array}{cc}
    b_1 & c_1\\[2mm]
    b_2 & c_2\\
    \end{array}
    \right|.
    \end{equation*}

    If we avoid the case
    $\frac{a_1}{a_2}=\frac{b_1}{b_2}=\frac{c_1}{c_2}$ then we have the following:
    \begin{enumerate}
        \item[(i)] $Q$ is elliptic if and only if $\Delta >0;$

        \item[(ii)] $Q$ is parabolic if and only if $\Delta =0;$

        \item[(iii)] $Q$ is hyperbolic if and only if $\Delta <0.$
    \end{enumerate}

    Using this argument one can construct concrete examples:
    \begin{enumerate}
        \item[(a)] $Q_1(x)=(x^2_1, x^2_2)$ is elliptic;

        \item[(b)] $Q_2(x)=(x^2_1, x_1x_2)$ is parabolic;

        \item[(c)] $Q_3(x)=(x^2_1-x^2_2, x_1x_2)$ is hyperbolic.
    \end{enumerate}
\end{example}

\begin{example}[The Stein--Ulam operator]

    Let us consider the following quadratic operator acting on $\br^3$:
    \begin{equation*}
    Q(x)=(x^2_1+2x_1x_2, x^2_2+2x_2x_3, x^2_3+2x_1x_3).
    \end{equation*}

    For a linear functional $f(x)=\l_1x_1+\l_2x_2+\l_3x_3$ we have the following quadratic form
    \begin{equation}\label{UlamLinearfunc}
    f(Q(x))=\l_1x^2_1+2\l_1x_1x_2+\l_2 x^2_2+2\l_2x_2x_3+\l_3
    x^2_3+2\l_3x_1x_3.
    \end{equation}

    The matrix of this quadratic form is
    \begin{equation*}
    A= \left(
    \begin{array}{ccc}
    \l_1 & \l_1 & \l_3\\[2mm]
    \l_1 & \l_2 & \l_2\\[2mm]
    \l_3 & \l_2 & \l_3\\
    \end{array}
    \right).
    \end{equation*}
    Due to Silvester's criterion, the quadratic form \eqref{UlamLinearfunc} is positive defined if and only if
    \begin{equation*}
    \l_1>0, \ \l_2>0, \ \l_3>0, \ \l_2>\l_1, \ \l_3>\l_2, \l_1>\l_3.
    \end{equation*}
    However, this system of inequalities has no solutions. Therefore, this quadratic operator is not elliptic.

    On the other hand if we take the linear functional $f:\br^3\to\br$ as $f(x)=x_1+x_2+x_3$  then we have $f(Q(x))=(x_1+x_2+x_3)^2\geq 0.$ Consequently, $Q:\br^3\to \br^3$ is the parabolic quadratic operator.
\end{example}

\begin{example} Let us consider the following quadratic operator $Q:\br^n\to \br^n$
    \begin{equation*}
    Q(x)=(x^2_1+x^2_n, x^2_2+x^2_n,\cdots, x^2_{n-1}+x^2_n,
    2x_n(x_1+\cdots+x_{n-1})),
    \end{equation*}
    where $x=(x_1,\cdots,x_n).$ Then for the linear functional $f(x)=\l_1x_1+\cdots+\l_nx_n$ we get the following quadratic form
    \begin{equation}\label{ExampleonR^n}
    f(Q(x))=\sum_{i=1}^{n-1}\l_ix^2_i+\left(\sum_{i=1}^{n-1}\l_i\right)x^2_n+
    2\l_nx_n(x_1+\cdots+x_{n-1}).
    \end{equation}

    The matrix of this quadratic form \eqref{ExampleonR^n} is
    \begin{equation*}
    A= \left(
    \begin{array}{cccccc}
    \l_1 & 0 & 0 & \cdots & 0 & \l_n\\[2mm]
    0    & \l_2 & 0 &\cdots  & 0 & \l_n\\[2mm]
    \cdot & \cdot & \cdot & \cdots  & \cdot & \cdot\\[2mm]
    0 & 0 & 0 & \cdots & \l_{n-1} & \l_n\\[2mm]
    \l_n & \l_n & \l_n & \cdots & \l_n & \sum\limits_{i=1}^{n-1}\l_i\\
    \end{array}
    \right).
    \end{equation*}
    Due to Silvester's criterion, the quadratic form \eqref{ExampleonR^n} is positive defined if and only if
    \begin{eqnarray*}
        \sum_{i=1}^{n-1}\l_i-\l_n^2\sum_{i=1}^{n-1}\frac{1}{\l_i}>0,
        \quad \l_i>0, \ i=\overline{1,n-1}.\
    \end{eqnarray*}
    Therefore, we obtain that
    \begin{equation*}
    K'_Q=\bigg\{f=(\l_1,\cdots,\l_n):\l_i>0, \ i=\overline{1,n-1}, \
    \l_n^2<\frac{\l_1+\l_2+\cdots+\l_{n-1}}
    {\frac{1}{\l_1}+\cdots+\frac{1}{\l_{n-1}}}\bigg\}.
    \end{equation*}
    Consequently, the given quadratic operator $Q$ is elliptic.
\end{example}

One can easily prove the following statement.

\begin{proposition} Let $Q:\br^n\to\br^n$ be a quadratic operator. Then the following assertions hold true:
    \begin{itemize}
        \item[(i)] If one of matrices $A_1,\cdots ,A_n$ is positive defined then $Q$ is an EQO.
        \item[(ii)] If matrices $A_1,\cdots,A_n$ are linear independent in the matrix algebra and
        commute each other then $Q$ is an EQO and $K'_Q$ is a minihedral cone.
    \end{itemize}
\end{proposition}


\section{The necessary condition for an existence of solutions}


In this section, we will consider elliptic operator equation and we will provide some necessary conditions for the existence of its solution.

The following equation
\begin{equation}\label{1}
    P(x)\equiv Q(x)+Ax+b=0, \quad x\in \br^n,
\end{equation}
is called \textit{an elliptic quadratic operator equation}, where $Q:\br^n\to \br^n$ is an elliptic quadratic operator,
$A:\br^n\to \br^n$ is a linear operator and $b\in \br^n$ is a given vector.

Let $K'_Q$ be an open convex cone given by \eqref{K'Q} associated with an elliptical operator
$Q$. For every $f\in K'_Q$ we denote by
\begin{equation*}
    \ce_f=\{x\in\br^n: f(Q(x)+Ax+b)\leq 0\},
\end{equation*}
and  it is  called an {\it ellipsoid} corresponding to $f\in K'_Q.$ It is obvious that if $\ce_f=\emptyset$ for some linear functional $f\in K'_Q$ then the elliptic operator equation \eqref{1} does not have any solutions.

Therefore,  the necessary condition for the solvability of the elliptic operator equation \eqref{GeneralQOE} is that $\ce_f\ne \emptyset$ for any $f\in K'_Q.$ Throughout this paper, we always assume that $\ce_f\ne \emptyset$ for any $f\in K'_Q.$

We define the following set  $$\ce_n(Q,A,b)\equiv\bigcap\limits_{f\in K'_Q}\ce_f.$$

\begin{theorem}\label{t1}
    If the equation \eqref{1} is solvable then  $\ce_n(Q,A,b)\ne \emptyset.$
\end{theorem}

We will prove this theorem after several auxiliary lemmas.


\subsection{Auxiliary results}


Recall that a set $M\subset \br^n$ is called {\it uniformly convex}
if for any $\e>0$ there exists $\d>0$ such that it follows from $x,y\in M,$
$\|x-y\|\geq \e$ and $\|z-\frac{x+y}{2}\|\leq \d$  that
$z\in M.$

\begin{lemma}\label{l1}
    $\ce_f$ is a closed bounded and uniformly convex set.
\end{lemma}

\begin{proof} Since $Q, A, f$ are continuous mappings, $\ce_f$ is a closed set.

    {\textsf{Boundedness.}} Note that if we show that $$f(Q(x)+Ax+b) >0$$
    for all $x\in\br^n$ such that $\|x\|\geq C$ for some $C>0$ then one has that $\ce_f\subset\bb(\theta,C)$. This means that $\ce_f$ is bounded, where
    $$
    \bb(a,r)=\{x\in\br^n:\ \|x-a\|\leq r\}, \ \ a\in\br^n, \ r>0.
    $$

    Without loss of generality, we may assume that $\|f\|=1.$ Then, we have that
    \begin{eqnarray*}
        f(Q(x)+Ax+b)&\geq& f(Q(x))-\|A\|\cdot \|x\|-\|b\|\geq \a \|x\|^2-\|A\|\cdot \|x\|-\|b\|.
    \end{eqnarray*}
    If
    \begin{equation*}
    \|x\|>\frac{\|A\|+\sqrt{\|A\|^2+4\a\cdot\|b\|}}{2\a},
    \end{equation*}
    then $\a \|x\|^2-\|A\|\cdot \|x\|-\|b\|>0$, i.e. $f(Q(x)+Ax+b)>0.$ Therefore, $\ce_f\subset\bb(\theta,C)$ where $C=\frac{\|A\|+\sqrt{\|A\|^2+4\a\cdot\|b\|}}{2\a}.$ Consequently, $\ce_f$ is bounded.

    {\sf Uniformly convexity}. We denote by  $P(x)\equiv Q(x)+Ax+b.$ Let $x,y\in \ce_f$
    with $\|x-y\|\geq \e >0$

    We consider a function
    \begin{equation*}
    \varphi(\l)=f(P(\l x+(1-\l)y))+\a\e^2\l(1-\l),
    \end{equation*}
    where $0\leq \l \leq 1.$ Since
    \begin{equation*}
    Q(\l x+(1-\l)y)=\l^2Q(x)+2\l(1-\l)B(x,y)+(1-\l)^2Q(y),
    \end{equation*}
    where $B(\cdot,\cdot)$ is a symmetric bilinear operator generated by $Q$
    then one gets that
    \begin{equation*}
    \frac{d^2\varphi}{d\l^2}(\l)=2f(Q(x-y))-2\a\e^2\geq
    2\a\|x-y\|^2-2\a\e^2\geq 0.
    \end{equation*}

    Thus $\varphi(\l)$ is a convex function. Due to Jensen's inequality
    it follows that
    \begin{equation*}
    \varphi\bigg(\frac{1}{2}\bigg)\leq
    \frac{1}{2}(\varphi(0)+\varphi(1)),
    \end{equation*}
    which yields that
    \begin{equation*}
    f\bigg(P\bigg(\frac{x+y}{2}\bigg)\bigg)+\frac{\a\e^2}{4}\leq
    \frac{1}{2}[f(P(x))+f(P(y))].
    \end{equation*}

    Since $x,y\in \ce_f,$ one has that $f(P(x))\leq 0$ and $f(P(y))\leq 0.$
    Hence, we obtain that
    \begin{equation}\label{2}
    f\bigg(P\bigg(\frac{x+y}{2}\bigg)\bigg)\leq -\frac{\a\e^2}{4}<0.
    \end{equation}

    On the other hand, we have that
    \begin{equation}\label{3}
    P\bigg(\frac{x+y}{2}+h\bigg)=P\bigg(\frac{x+y}{2}\bigg)+2B\bigg(\frac{x+y}{2},
    h\bigg)+Q(h)+Ah.
    \end{equation}
    Since $\ce_f$ is a bounded set one can choose $C>0$ such that
    \begin{equation*}
    \bigg\|B\bigg(\frac{x+y}{2},h\bigg)\bigg\|\leq C\cdot \|h\|, \ \
    \forall x,y\in \ce_f.
    \end{equation*}
    Therefore, there exists $\d>0$ (which does not depend on $x,y\in
    \ce_f$) such that
    \begin{equation}\label{4}
    f\bigg(2B\bigg(\frac{x+y}{2}, h\bigg)+Q(h)+Ah\bigg)\leq
    \frac{\a\e^2}{4},
    \end{equation}
    for all $\|h\|\leq \d.$ Then it follows from \eqref{2}, \eqref{3}, and \eqref{4} that
    \begin{equation*}
    f\bigg(P\bigg(\frac{x+y}{2}+h\bigg)\bigg)\leq 0,
    \end{equation*}
    whence
    \begin{equation*}
    \bb \bigg(\frac{x+y}{2},\d\bigg)\subset \ce_f
    \end{equation*}
    which follows from $x,y\in \ce_f,$
    $\|x-y\|\geq \e$ and $\|z-\frac{x+y}{2}\|\leq \d$ that
    $z\in \ce_f.$
\end{proof}

\begin{lemma}\label{l2}
    The map $f\to \ce_f $  is continuous in $K'_Q$, i.e. if $f_0\in
    K'_Q$ then for any $\e>0$ there exists $\d=\d(\e, f_0)>0$ such that
    from $\|f-f_0\|\leq \d$ and $f\in K'_Q$ it follows that
    $\ce_f\subset U_\e(\ce_{f_0}),$ where
    \begin{equation*}
    U_\e(\ce_{f_0})=\{x\in\br^n: \inf_{y\in \ce_{f_0}}\|x-y\|<\e\}.
    \end{equation*}
\end{lemma}

\begin{proof} Without loss of generality, we may assume that
    $\|f\|=\|f_0\|=1.$ Let
    \begin{equation}\label{1e1}
    f_0(Q(x))\geq \a\|x\|^2, \ \|Q(x)\|\leq M\cdot \|x\|^2, \ \
    \d_1=\frac{\a}{2M}.
    \end{equation}
    Then, for any $f\in K_Q'$ with $\|f-f_0\|\leq \d_1$ it follows from
    \eqref{1e1} that
    \begin{equation}\label{1e2}
    f(Q(x))=f_0(Q(x))+(f-f_0)(Q(x))\geq \frac{\a}{2}\|x\|^2.
    \end{equation}

    The inequality \eqref{1e2} implies that
    \begin{equation*}
    f(P(x))\geq \frac{\a}{2}\|x\|^2-\|A\|\cdot \|x\|-\|b\|.
    \end{equation*}
    Therefore, if $\|x\|>C,$ where $C=\frac{\|A\|+\sqrt{\|A\|^2+2\a\cdot\|b\|}}{\a},$ then $f(P(x))>0$ for
    any $f\in \bb(f_0,\d_1)$ This means that
    $\bigcup\limits_{f\in \bb(f_0,\d_1)}\ce_f$ is bounded. Consequently, there exists $L>0$ such that $\|P(x)\|\le L$ for any
    $x\in \bigcup\limits_{f\in\bb(f_0,\d_1)}\ce_f.$ Now we show that there is  $\eta>0$ such that
    \begin{equation}\label{5}
    \{x\in\br^n: f_0(P(x))\leq\eta\} \subset U_\e(\ce_{f_0}).
    \end{equation}
    Indeed, let $x\notin U_\e(\ce_{f_0}).$ Since $\ce_{f_0}$ is a closed convex set, there exists a projection $x_0$ of $x$ onto $\ce_{f_0}.$ Then $x_0$ should be a boundary point of $\ce_{f_0}$. Hence $f_0(P(x_0))=0$.

    By letting $h=x-x_0$ and using Taylor formula, one can get that
    \begin{equation}\label{6}
    f_0(P(x))=f_0(P(x_0))+f_0(P'(x_0)(h))+f_0(Q(h)),
    \end{equation}
    where $P'(x_0)=2B(x_0,h)+Ah$ is the Frechet derivative of the operator $P$ at
    $x_0.$ Since $x_0$ is a projection of $x$ onto $\ce_{f_0}$ one has that
    $f_0(P(x_0))=0.$ Therefore, we have that
    \begin{equation}\label{65}
    f_0(P(x))=f_0(P'(x_0)(h))+f_0(Q(h)).
    \end{equation}

    Since $f_0\in K_Q^{'}$, one obtains that $f_0(P(x))\geq0$ and $f_0(Q(h))\geq0$ for any $x\notin\ce_{f_0}$ and for all
    $h\in\br^n.$  Let $x'=x_0+th$ and $0<t<1.$ Then, it follows from \eqref{65} and $x'\notin\ce_{f_0}$ that
    \begin{equation*}
    f_0(P(x'))=tf_0(P'(x_0)(h))+t^2f_0(Q(h))>0,
    \end{equation*}
    for any $0<t<1.$ This yeilds that $f_0(P'(x_0)(h))\geq 0.$

By means of $\ f_0(P'(x_0)(h))\geq 0$ and
    $\|x-x_0\|=\|h\|>\e$, we obtain from \eqref{65} that
    \begin{equation*}
    f_0(P(x))\geq f_0(Q(h))>\a\e^2.
    \end{equation*}
    Therefore, if $0<\eta\leq\a\e^2$ then it
    follows from $x\notin U_\e(\ce_{f_0})$  that $f_0(P(x))>\a\cdot \e^2\geq \eta.$ In other words \eqref{5}
    holds true. We denote by $\d=\min(\d_1, \frac{\a\e^2}{L}).$ We now check that $\ce_f \subset U_\e(\ce_{f_0})$ for any $f\in \bb(f_0,\d).$ Indeed, for $\eta=L\d$ and $x\in
    \ce_f$ we have that
    \begin{eqnarray*}
        f_0(P(x))&=&f(P(x))+(f_0-f)(P(x))\leq (f_0-f)(P(x))\leq \|f-f_0\|\cdot \|P(x)\|\leq \d L=\eta.
    \end{eqnarray*}
    Hence
    \begin{equation*}
    \ce_f\subset \{x: f_0(P(x))\leq \eta\} \subset U_\e(\ce_{f_0}),
    \end{equation*}
    and this completes the proof.
\end{proof}

\begin{lemma}\label{l3} Let  $f_0,f_1\in K'_Q$ and $\ce_{f_0}, \ce_{f_1}$ be the corresponding ellipsoids. Then
    $\ce_{f_0}\cap\ce_{f_1}\ne \emptyset.$
\end{lemma}
\begin{proof} We assume that $\ce_{f_0}\cap\ce_{f_1}=\emptyset.$ Let
    $f_\l=\l f_1+(1-\l)f_0,$ $0\leq \l\leq 1$ and $\ce_{f_\l}$ be the corresponding
    ellipsoid  to $f_\l\in K'_Q$. It is worth noting that
    $\ce_{f_\l}\subset \ce_{f_0}\cup\ce_{f_1}.$ Indeed, if $x\in
    \ce_{f_\l}$ then
    \begin{equation}\label{17}
    f_\l(P(x))=\l f_1(P(x))+(1-\l)f_0(P(x))\leq 0.
    \end{equation}
    It follows from \eqref{17} that
    $f_1(P(x))\leq 0$ or $f_0(P(x))\leq 0$ which means $x\in
    \ce_{f_0}\cup\ce_{f_1}$.

Since $\ce_{f_\l}$ is a convex (connected) set and $\ce_{f_\l}\subset \ce_{f_0}\cup\ce_{f_1}$
    with  $\ce_{f_0}\cap\ce_{f_1}=\emptyset$, One has that either $\ce_{f_\l}\subset \ce_{f_0}$ or $\ce_{f_\l}\subset \ce_{f_1}.$ We denote by $I_0=\{\l: \ce_{f_\l}\subset \ce_{f_0}\}$ and  $I_1=\{\l:
    \ce_{f_\l}\subset \ce_{f_1}\}.$ It is clear that $I_0\cup
    I_1=[0,1],$ $I_0\cap I_1=\emptyset.$ Since $0\in I_0,$ $1\in I_1,$ one has $I_0\ne \emptyset$ and $I_1\ne
    \emptyset.$

    We know that  $\ce_{f_0}$ and  $\ce_{f_1}$  compact sets and $\ce_{f_0}\cap\ce_{f_1}=\emptyset$. Then there exists $\e >0$  such that   $U_\e(\ce_{f_0})\cap U_\e(\ce_{f_1})=\emptyset.$ It follows from Lemma \ref{l2}  that $I_0$ and $I_1$ are open subsets of $[0,1].$ Therefore, $[0,1]$ is a union of two disjoint nonempty open sets $I_0,I_1$ which is a contradiction. Consequently, any
    two ellipsoids have a nonempty intersection.
\end{proof}

\begin{lemma}\label{l4} Let  $f_0,f_1\in K'_Q$ and  $\ce_{f_0}, \ce_{f_1}$
    be the corresponding ellipsoids and $\Delta:=\ce_{f_0}\cap\ce_{f_1}\neq\emptyset.$ If a hyperplane $H$ does not
    intersect the set $\Delta$, i.e. $H\cap\Delta=\emptyset$, then there exists $f\in co(f_0,f_1)$ such
    that $\ce_f\cap H=\emptyset,$ where $\ce_f$ is the ellipsoid
    corresponding to $f$.
\end{lemma}

\begin{proof} We set that $\Delta_i=\ce_{f_i}\cap H,$ $i=0,1$ and assume that  $\Delta_i\ne
    \emptyset,$ $i=0,1,$ otherwise the proof is trivial. Suppose the contrary, i.e.
    $\ce_{f_\l}\cap H\ne \emptyset$ for all $f_\l=\l f_1+(1-\l)f_0,$
    $\l\in [0,1].$ By definition we have that
    \begin{equation*}
    \ce_{f_\l}=\{x\in\br^n: \l f_1(P(x))+(1-\l)f_0(P(x))\leq 0\}.
    \end{equation*}

    Hence, for any $\l\in[0,1]$, one can get that
    \begin{equation}\label{18}
    \Delta=\ce_{f_0}\cap\ce_{f_1}\subset\ce_{f_\l}\subset\ce_{f_0}\cup\ce_{f_1}.
    \end{equation}

    Now we will prove that one of the sets $\ce_{f_\l}\cap\Delta_0$,
    $\ce_{f_\l}\cap\Delta_1$ is empty and another one is nonempty. We assume that both sets are simultaneously either nonempty or empty. If they are nonempty we then consider a closed segment $[x,y],$ where $x\in \ce_{f_\l}\cap\Delta_0,$
    $y\in \ce_{f_\l}\cap\Delta_1,$ and $[x,y]=\{\m
    x+(1-\m)y:\ \m\in[0,1]\}.$ It follows from \eqref{18} that
    \begin{equation}\label{19}
    [x,y]\subset \ce_{f_\l}\subset\ce_{f_0}\cup\ce_{f_1}, \ [x,y]
    \subset\Delta_0\cup\Delta_1.
    \end{equation}
    On the other hand
    \begin{equation*}
    \Delta_0\cap\Delta_1=(\ce_{f_0}\cap H)\cap (\ce_{f_1}\cap
    H)=\Delta\cap H=\emptyset.
    \end{equation*}
    It yields that $\Delta_0\cup\Delta_1$ cannot contain the segment $[x,y]$ which
    contradicts \eqref{19}.

    Moreover,  it follows from \eqref{18} that
    \begin{eqnarray*}
    \ce_{f_\l}\cap H\subset(\ce_{f_0}\cap H)\cup(\ce_{f_1}\cap H)&=&\Delta_0\cup\Delta_1, \\
    (\ce_{f_\l}\cap\Delta_0)\cup(\ce_{f_\l}\cap\Delta_1)&=&\ce_{f_\l}\cap(\Delta_0\cup\Delta_1)\supset\ce_{f_\l}\cap H\neq\emptyset,
    \end{eqnarray*}
    for any $\l\in [0,1].$ This shows that both sets $\ce_{f_\l}\cap\Delta_0$, $\ce_{f_\l}\cap\Delta_1$ cannot also be empty.

    Thus, for any $\l\in [0,1]$ we have either
    $\ce_{f_\l}\cap\Delta_0\ne\emptyset$ or
    $\ce_{f_\l}\cap\Delta_1\ne\emptyset.$ We denote by
    \begin{equation*}
    I_0=\{\l: \ce_{f_\l}\cap\Delta_0\ne \emptyset\}, \
    I_1=\{\l: \ce_{f_\l}\cap\Delta_1\ne \emptyset\}.
    \end{equation*}
    It is clear that $I_0\cup I_1=[0,1],$ $I_0\cap I_1=\emptyset.$ Since $0\in I_0,$ $1\in I_1,$ one has $I_0\ne \emptyset$ and $I_1\ne
    \emptyset.$

    Since  $\ce_{f_0}$ and  $\ce_{f_1}$ are compact sets and $\Delta_0\cap\Delta_1=\emptyset,$ there exists $\e >0$  such that
    $U_\e(\Delta_0)\cap U_\e(\Delta_1)=\emptyset.$
    Then due to Lemma \ref{l2}, $I_0$ and $I_1$ are open subsets of $[0,1].$
    Therefore, $[0,1]$ is a union of two disjoint nonempty open sets $I_0,I_1$ which is a contradiction. Consequently,  for some $f\in co(f_0,f_1)$ we must have that $\ce_f\cap
    H=\emptyset.$
\end{proof}

\begin{lemma}\label{l5} Let $f_i\in K'_Q,$ $i=\overline{1,n}$ and $\ce_{f_i},\
    i=\overline{1,n}$ be the corresponding ellipsoids. Let
    $\Delta=\bigcap\limits_{i=1}^n\ce_{f_i}\ne\emptyset$ and $H\cap
    \Delta=\emptyset$ for some hyperplane $H$. Then there exists
    $f\in co(f_1,\dots,f_n)$ such that $\ce_f\cap H=\emptyset.$
\end{lemma}
\begin{proof} We will use the mathematical induction with respect to $n$. For $n=1$ the
    assertion is trivial. For $n=2$ the assertion was proven by Lemma
    \ref{l4}. We assume that the assertion of the lemma is true for $n=k-1$ and we prove it for $n=k$ Denote
    \begin{equation*}
    \Delta_{k-1}=\bigcap_{i=1}^{k-1}\ce_{f_i},  \ \ \ B=\ce_{f_k}\cap H.
    \end{equation*}

    Since $\Delta_{k-1},$ $B$ are compact sets  and $\Delta_{k-1}\cap B=\Delta\cap H=\emptyset,$ we can strictly separate the sets
    $\Delta_{k-1}$ and $B$ by some hyperplane $L,$ i.e.,
    \begin{equation*}
    \Delta_{k-1}\cap L=\emptyset, \ B\cap L=\emptyset.
    \end{equation*}
Thus, $\Delta_{k-1}$ and $B$ lie in different semi-spaces   defined by the hyperplane $L$.
    Since $L\cap \Delta_{k-1}=\emptyset$, due to assumption of the mathematical induction,
    there exists ${\hat f}\in co(f_1,\dots,f_{k-1})$ such that
    \begin{equation}\label{10}
    \ce_{{\hat f}}\cap L=\emptyset.
    \end{equation}

    If $x\in \Delta_{k-1}$ then $f_i(P(x))\leq 0$ for any
    $i=\overline{1,k-1}$ and $x\in \ce_{{\hat f}}$, i.e.
    $\Delta_{k-1}\subset \ce_{{\hat f}}.$ Since $\ce_{{\hat f}}$  is a convex set, it follows from \eqref{10} that $\ce_{{\hat f}}$ lies
    in the same semi-space where $\Delta_{k-1}$ is located. Thus,
    $\ce_{{\hat f}}\cap B=\emptyset$ i.e.
    \begin{equation}\label{11}
    \ce_{{\hat f}}\cap B=\ce_{{\hat f}}\cap (\ce_{f_k}\cap
    H)=(\ce_{{\hat f}}\cap \ce_{f_k})\cap H=\emptyset.
    \end{equation}

    According to Lemma \ref{l4} (in the case $n=2$), there exists $f\in co({\hat f},f_k)\subset
    co(f_1,..,f_k)$ such that $\ce_f\cap H=\emptyset.$ This completes the proof.
\end{proof}

\subsection{The proof of Theorem \ref{t1}.}


First of all,  we will show
that  an intersection of any finite numbers of ellipsoids is
nonempty. Let $k$ be a minimal  number such that
an intersection of any $k-1$ ellipsoids is nonempty and there are
$k$ ellipsoids $\ce_{f_1},\cdots,\ce_{f_k}$ with an empty
intersection. Due to Lemma \ref{l3} we have $k\geq 3$.  Let
$\Delta_{k-1}=\bigcap\limits^{k-1}_{i=1}\ce_{f_i}$. Then $\Delta_{k-1}\ne
\emptyset$ and $\Delta_{k-1}\cap \ce_{f_k}=\emptyset.$
Therefore, since $\Delta_{k-1}$ and $\ce_k$ compact sets, there exists some hyperplane $H$ strictly separating
$\Delta_{k-1}$ and $\ce_k.$ Since $\Delta_{k-1}\cap H=\emptyset$, then due to Lemma \ref{5},
there exists $\ce_f$ such that $\ce_f\cap H=\emptyset$, moreover
$\ce_f$ and $\Delta_{k-1}$ lie  in the same semi-space. Hence
$\ce_f$ and $\ce_{f_k}$ are located in different semi-spaces, i.e., $\ce_f\cap \ce_k=\emptyset.$ It contradicts
to the assertion of Lemma \ref{l3}. Therefore, an intersection of any finite number of
ellipsoids is nonempty. Consequently,
\begin{equation*}
    \ca=\{\ce_f: f\in K'_Q\}
\end{equation*}
is a centered family of compact sets. Hence
\begin{equation}\label{12}
\ce_n(Q,A,b)=\bigcap_{f\in K'_Q}\ce_f\ne\emptyset
\end{equation}
which completes the proof of Theorem \ref{t1}.

\begin{remark}\label{efimplye}
    It follows from the proof of Theorem \ref{t1} that if $\ce_f\ne\emptyset$ for any $f\in K'_Q$ then $\ce_n(Q,A,b)=\bigcap\limits_{f\in K'_Q}\ce_f\ne\emptyset.$
\end{remark}

It is clear that $\ce_n(Q,A,b)$ is a convex compact set. Due to Krein-Milman theorem, the set $\rm{{\textbf{Extr}}(\ce_n(Q,A,b))}$ of extreme points of the set $\ce_n(Q,A,b)$ is nonempty. Let $$X_n(Q,A,b)=\{x\in\br^n: P(x)\equiv Q(x)+Ax+b=0\}$$
be a set of solutions of equation \eqref{1}. Since $X_n(Q,A,b)\subset\ce_f$
for any $f\in K'_Q$ we have
\begin{equation}\label{13}
    X_n(Q,A,b)\subset \ce_n(Q,A,b)
\end{equation}

The following theorem gives more precise description of the set $X_n(Q,A,b).$

\begin{theorem}\label{t2}
    Any solution  of the equation \eqref{1} is an
    extremal point of the set $\ce_n(Q,A,b).$
\end{theorem}
\begin{proof} We suppose the contrary, i.e.,  $P(x_0)=0$ and  $x_0=\l
    x _1+(1-\l)x_2,$ where $x_1,x_2\in \ce_n(Q,A,b),$ $x_0\neq x_1,$ $x_0\neq x_2,$ and $\l\in (0,1).$ If
    $f\in K'_Q$ then $f(P(x_1))\leq 0$ and $f(P(x_2))\leq 0.$ Since
    $0<\l<1$ and $\ce_f$ is a uniformly convex set, we have that $x_0\in {\rm
        int}\ce_f,$ i.e. $f(P(x_0))<0.$ We then get that $P(x_0)\ne 0$ and it contradicts to $x_0\in X_n(Q,A,b).$ Therefore,
$X_n(Q,A,b)\subset \textbf{Extr}(\ce_n(Q,A,b)).$ This completes the proof.
\end{proof}

\subsection{Some examples}


In general, the condition $\ce_n(Q,A,b)\ne\emptyset$ does not imply an existence of the solutions of the equation \eqref{1}.

\begin{example}
    Let us consider the following equation in $\br^3$
    \begin{equation}\label{15}
    \left\{\begin{array}{lll}
    x^2_1-x_2=0\\
    x^2_2-2=0\\
    x^2_3+x_2-1=0.
    \end{array}\right.
    \end{equation}
    It is easy to see if $Q(x)=(x_1^2, x^2_2,x^2_3),$ $Ax=(-x_2,0,x_2),$ $b=(0,-2,-1)$ then equation \eqref{15} is elliptic operator equation.
    We have that $K'_Q=\{(\xi_1,\xi_2,\xi_3): \ \xi_1>0, \ \xi_2>0, \ \xi_3>0\}$. The condition $\ce_n(Q,A,b)\ne\emptyset$ is also satisfied because of $(0,0,0)\in \ce_f$ for any
    $f\in K'_Q.$  On the other hand, it is evident that equation
    \eqref{15} has no solutions.
\end{example}


\section{The sufficient condition for an existence of solutions}


\begin{theorem}\label{criterion}
    The elliptic operator equation \eqref{1} is solvable if and only if $\bigcap\limits_{f\in K'_Q}\partial \ce_f\neq \emptyset.$  Moreover, if the elliptic operator equation \eqref{1} is solvable  then $X_n(Q,A,b)= \bigcap\limits_{f\in K'_Q}\partial \ce_f$.
\end{theorem}
\begin{proof}
    It is clear that $X_n(Q,A,b)\subset \partial \ce_f$ and $X_n(Q,A,b)\subset \bigcap\limits_{f\in K'}\partial \ce_f.$ Since $K'_Q$ is an open
    cone, it follows from $f(P(x))=0$ for all $f\in K'_Q$ that $P(x)=0.$
    Hence $\bigcap\limits_{f\in K'}\partial \ce_f\subset X_n(Q,A,b).$ Thus, we obtain that
$X_n(Q,A,b)=\bigcap_{f\in K'_Q}\partial \ce_f.$
\end{proof}

In this section, we provide some sufficient conditions to insure an existence of solutions of the equation \eqref{1}.

\subsection{The lower dimensional space}


For small dimensions ($n=1,2$), the condition $\ce_n(Q,A,b)\ne\emptyset$ remains to be
the sufficient condition for the existence of solutions of the equation \eqref{1}.

\begin{theorem}\label{t3}
    Let $n\leq 2$. If $\ce_n(Q,A,b)\ne\emptyset$ is satisfied then $X_n(Q,A,b)\ne \emptyset.$
\end{theorem}

\begin{proof} Let $n=1.$ Then the equation \eqref{1} has the following form
    \begin{equation*}
    ax^2+bx+c=0,\qquad a\ne 0.
    \end{equation*}
    Note that the condition $\ce_n(Q,A,b)\ne\emptyset$ is equivalent to the condition $b^2-4ac\geq 0.$ One
    can see that
    \begin{equation*}
\ce_n(Q,A,b)=\{x: {\rm sign} (a(ax^2+bx+c))\leq 0\},\qquad X_n(Q,A,b)=\textbf{\textup{Extr}}\ce_n(Q,A,b)\ne
    \emptyset.
    \end{equation*}

    Let $n=2$. We assume the contrary, i.e. $X_n(Q,A,b)=\emptyset$. If
    $f_1,f_2\in K'_Q$ are linearly independent and $\ce_{f_1},
    \ce_{f_2}$ are the corresponding ellipsoids, then
    \begin{equation*}
    X_n(Q,A,b)=\partial \ce_{f_1}\cap\partial \ce_{f_2}=\emptyset.
    \end{equation*}
    On the other hand it follows from Theorem \ref{t1} that
    $\ce_{f_1}\cap\ce_{f_2}\ne \emptyset$. Hence, one ellipsoid
    lies inside another one. Without loss any generality, we suppose that $\ce_{f_1}\subset \ce_{f_2}$.
    Since $K'_Q$ is an open cone, there exists a sufficiently small number $\e>0$ such that $f_1-\e f_2\in K_Q'.$ Let
    $\ce_\e$ be the ellipsoid corresponding to $f_1-\e f_2$.    Now we will show that $\ce_\e\subset \ce_{f_1}$. Indeed, if $x\in
    \ce_{f_2}\setminus\ce_{f_1}$ then $f_1(P(x))>0$ and $f_2(P(x))\leq
    0$. Hence, we obtain that
    \begin{equation*} (f_1-\e f_2)(P(x))=f_1(P(x))-\e
    f_2(P(x))>0,
    \end{equation*}
    i.e. the ellipsoid $\ce_\e$ lies either inside of $\ce_{f_1}$ or outside
    of $\ce_{f_2}$. The later one is impossible, since Lemma
    \ref{l3} implies that $\ce_{f_2}\cap \ce_{\e}\neq\emptyset$. Thus, we have that $\ce_\e\subset \ce_{f_1}.$ Therefore, there exists a number $\e_0>0$ such that
    \begin{equation*}
    f_1-\e_0 f_2\notin K'_Q, \ f_1-\e_0 f_2\in \partial K'_Q, \ f_1-\e_0
    f_2\ne 0,
    \end{equation*}
    and  one has that $\ce_\e\subset \ce_{f_1}$ for  any $\e\in [0,\e_0)$. It this case, there exists a number $d>0$ such that
    \begin{equation}\label{165}
    {\rm diam}\{x: (f_1-\e f_2)(P(x))\leq 0\}\leq d,
    \end{equation}
    for any $\e\in
    [0,\e_0).$ By taking limit from \eqref{165} whenever $\e\to\e_0$ one has that
    \begin{equation}\label{170}
    {\rm diam}\{x: (f_1-\e_0 f_2)(P(x))\leq 0\}\leq d.
    \end{equation}
    Since $\ce_\e\supset\bigcap\limits_{f\in K'}\ce_f=\ce_n(Q,A,b)\ne\emptyset$, the set $\{x: (f_1-\e_0 f_2)(P(x))\leq 0\}\supset\ce_\e$ is not empty.

    On other hand, since $f_1-\e_0 f_2\notin K'_Q,$ the set $\{x: (f_1-\e_0
    f_2)(P(x))\leq 0\}$ cannot be an ellipsoid. We know that  all the second order curves except the ellipsoid are  unbounded. Therefore, we obtain that
    \begin{equation}\label{180}
    {\rm diam}\{x: (f_1-\e_0 f_2)(P(x))\leq 0\}=+\infty,
    \end{equation}
    and this contradicts to \eqref{170}. This completes the proof.
\end{proof}

\subsection{The higher dimensional space}


In what follows, we will consider the case when $n\geq 3.$ In order to avoid some technical calculations, we always suppose that the following conditions are satisfied
\begin{equation}\label{190}
{\overline K}'_Q\cap (-{\overline K}'_Q)=\{0\}, \qquad {\rm int}\ce_n(Q,A,b)\ne
\emptyset,
\end{equation}
where, ${\overline K}'_Q$ is a closure of the cone $K'_Q$ in a norm topology and ${\rm{int}}\ce_n(Q,A,b)$ is an interior of the set $\ce_n(Q,A,b).$

\begin{remark}  By lowering the degree  $n$ of the space, one can always achieve to the condition \eqref{190}. For example, if $f\in {\overline K}'_Q\cap (-{\overline K}'_Q)$ and $f\ne 0$ then $f(Q(x))\geq 0$
    and $-f(Q(x))\geq 0,$ i.e. $f(Q(x))\equiv 0,$ and $ \ f(P(x))=f(Ax+b)=0.$
    By means of the last linear equation, one can exclude one variable from the elliptic operator equation and we get an elliptic operator equation in $\br^{n-1}.$  This procedure can be repeated until we get the condition
    ${\overline K}'_Q\cap (-{\overline K}'_Q)=\{0\}.$ Similarly, if ${\rm{int}}\ce_n(Q,A,b)=\emptyset$ then $\ce_n(Q,A,b)\subset \{x:\varphi(x)=c\},$ where $\varphi$ is a linear continuous
    functional and  $c$ is a number. Since all solutions of the elliptic operator equation \eqref{1} lie
    in $\ce_n(Q,A,b),$ the equation $P(x)=0$ is equivalent to system of
    equations $P(x)=0, \ \varphi(x)=c.$ By means of the linear equation
    $\varphi(x)=c,$ one can again obtain an elliptic operator equation in $\br^{n-1}.$ We can repeat this process until  we get the condition ${\rm{int}}\ce_n(Q,A,b)\ne \emptyset.$
\end{remark}

We recall some notions from the convex analysis. Let $K$ be a closed convex cone. \textit{A face} of a cone $K$ is a convex subset $K'$ of $K$ such that every (closed) segment in $K$ with a relative interior point in $K'$ has both end points in $K'.$ \textit{An extremal ray} is a face which a half-line emanating from the origin.

Let $\bk$ be the set of extremal rays of the closed cone
${\overline K}'_Q.$ We denote by
\begin{equation}\label{20}
\Pi_f=\{x: f(P(x))\leq 0, \quad f\in \bk\}
\end{equation}

\begin{lemma}\label{l}
    The set $\Pi_f$ is a closed convex unbounded set and
    $\ce_n(Q,A,b)\subset \Pi_f.$
\end{lemma}
\begin{proof} {\sf Closeness.} Since $Q,A, f$ are continuous mappings, $\Pi_f$ is a closed set.

    {\sf Convexity}. Let $x_1,x_2\in \pi_f$ and $0\leq \l \leq 1.$ Since $f(P(x_1))\leq 0, \ f(P(x_2))\leq 0$ and  $f(Q(x_1-x_2)) \geq 0,$ we obtain that
    \begin{eqnarray*}
        f(P(\l x_1+(1-\l)x_2)) &=& f(Q(\l x_1+(1-\l)x_2))+f(A(\l
        x_1+(1-\l)x_2))+b \\
        & = & \l f(P(x_1))+(1-\l)f(P(x_2))-\l(1-\l)f(Q(x_1-x_2))\leq 0.
    \end{eqnarray*}
    Consequently, $\Pi_f$ is a convex set.

    {{ Unboundedness.}} We assume the contrary i.e. $\Pi_f$ is bounded.

    Then for any $x_0$ with $\|x_0\|=1$ there exists $\lambda_0>0$, such that for all
    $|\l|>\l_0$ one has $\l x\notin \pi_f,$ i.e.,
    \begin{equation*}
    f(P(\l x_0))=\l^2f(Q(x_0))+\l f(Ax_0)+f(b)>0,
    \end{equation*}
    which implies that $f(Q(x_0)) > 0$. Hence, $f(Q(\cdot))$ is continuous and
    positive defined on the unit sphere. Since the unit sphere is a compact set, then $f(Q(x))\geq \a>0$ for any $x$ with $\|x\|=1$, i.e., $f(Q(x))\geq \a\|x\|^2$. This means $f\in K'_Q,$ which contradicts
    to $f\in \bk.$

    Now, we are going to prove that $\ce_n(Q,A,b)\subset \Pi_f.$ Let  $x\in \ce_n(Q,A,b)$. Since $f\in\bk\subset \overline{K}'_Q,$ we can choose a sequence $\{f_n\}\subset K'_Q$ such that  $f_n\to f$ in a norm topology. Let $\ce_{f_n}$ be the corresponding ellipsoids.
    It is clear that $\ce_n(Q,A,b)\subset \ce_{f_n},$ therefore,
    $f_n(P(x))\leq 0$. Consequently, $f(P(x))\leq 0$, it implies that $x\in \Pi_f.$ This completes the proof.
\end{proof}

\begin{theorem}\label{t4} One has that $\bigcap\limits_{f\in\bk}\Pi_f=\ce_n(Q,A,b).$
\end{theorem}

\begin{proof} It is clear that due to Lemma \ref{l}, one has that $\ce_n(Q,A,b)\subset \bigcap\limits_{f\in\bk}\Pi_f.$ We
    will prove the inverse inclusion. If $f\in K'_Q$ then according to the Krein-Milman and the Caratheodory theorems we have that
    \begin{equation}\label{121}
    f=\sum^m_{i=1}\a_if_i, \quad \a_i>0, \quad f_i\in \bk, \ m\leq n,
    \end{equation}
    where $n$ is a dimension of the space $\br^n.$

    If $x\notin \ce_n(Q,A,b)$ then $x\notin \ce_f$ for some $f\in K'_Q$.
    By means of \eqref{121}, one can obtain
    $f(P(x))=\sum_{i=1}^m\a_if_i(P(x))>0.$ Since $\a_i>0,$ $i=\overline{1,n},$  we get that
    $f_{i_0}(P(x))>0$  for some $1\leq i_0\leq n$ i.e. $x\notin \Pi_{f_{i_0}}$ for some $i_0.$
    Therefore, $x\notin \ce_n(Q,A,b)$ implies that $x\notin \bigcap\limits_{f\in
        \bk}\Pi_f$ and it means that $\bigcap\limits_{f\in \bk}\Pi_f\subset \ce_n(Q,A,b).$ This completes the proof.
\end{proof}

Let us recall some notions. Let $x_0$ be a boundary point of $\ce_n(Q,A,b)$ and $\gl_{x_0}$ be the
set of supporting hyperplanes  to $\ce_n(Q,A,b)$ at point $x_0.$  A point
$x_0$ is called {\it a boundary point of order} $k$ if the affine
dimension of $\bigcap\limits_{H\in \gl_{x_0}}H$ is equal to $k.$ In
particular, a boundary point of order 0 is called {\it a vertex} of
$\ce_n(Q,A,b).$ A connected component of the boundary points of order
$k$ is called a $k-$\textit{boundary} of the set $\ce_n(Q,A,b).$

\begin{theorem}\label{t5}
    If $x_0$ is a boundary point of order $k$ of the set $\ce_n(Q,A,b)$ then there
    are at least $n-k$ linear independent functionals
    $f_1,\cdots,f_{n-k}\in \bk$ such that $f_i(P(x_0))=0, \
    i=\overline{1,n-k}.$
\end{theorem}
\begin{proof}
    Let $H=\{x: \varphi(x)=c\}$ be a supporting hyperplane to $\ce_n(Q,A,b)$
    at the point $x_0$ and $\varphi(x)\leq c$ for any $x\in \ce_n(Q,A,b).$ Then for
    any arbitrary $\e>0$ we have that $\ce_n(Q,A,b)\cap H_\e=\emptyset$ where
    $H_\e=\{x: \varphi(x)=c+\e\}.$

    Since  $\ce_n(Q,A,b)$ is a compact set, there  exists $\d>0$ such that
    $U_\d(\ce_n(Q,A,b))\cap H_\e=\emptyset$ where $U_\d(\ce_n(Q,A,b))=\{x\in\br^n:
    \rho(x,\ce_n(Q,A,b))\leq \d\}.$

    Obviously, $\partial U_\d(\ce_n(Q,A,b))=\{x\in\br^n: \rho(x,\ce_n(Q,A,b))=\d\}$ is a compact set
    and a family of open sets $\{\br^n\setminus \ce_f\}, \
    f\in K'_Q$ is its open cover. Then it has a finite open subcover $\{\br^n\setminus \ce_{f_i}\}, \ i=\overline{1,r}.$ This implies that
    $\bigcap\limits^r_{i=1}\ce_{f_i}\subset U_\d(\ce_n(Q,A,b))$. Therefore, one has that
    \begin{equation}\label{122}
    \bigg(\bigcap^r_{i=1}\ce_{f_i}\bigg)\cap H_\e =\emptyset.
    \end{equation}
    for any $\e >0.$ It follows from \eqref{122} and Lemma \ref{l5}  that there exists
    $f\in K'_Q$ such that $\ce_f\cap H_\e=\emptyset,$ moreover, one can
    assume that $\|f\|=1.$ Hence, for any $\e>0$ there exists
    $f_\e\in K'_Q,$ $\|f_\e\|=1$ such that
    \begin{equation}\label{123}
    \ce_{f_\e}\cap H_\e = \emptyset.
    \end{equation}

    We take any sequence $\{\e_n\}\subset\br_+$ such that $\e_n \to 0$.
    Let $f_0\in \overline{K}'_Q$ be a limit point of $\{f_{\e_n}\}$. Since $x_0\in
    \ce_n(Q,A,b)$, it implies that $f_0(P(x))\leq 0.$ We will prove that $f_0(P(x_0))=0.$ Indeed, if $f_0(P(x_0))<0$ then $x_0$ belongs to
    an interior of the set $\Pi_{f_0}=\{x: f_0(P(x))\leq 0\}.$ Hence, the intersection
    $\Pi_{f_0}\cap H$ is stable with respect to small perturbations of $f_0$
    and $H.$ This contradicts to \eqref{122}. Therefore, we obtain that
    $f_0(P(x_0))=0$ and $H$ is a tangent hyperplane to $\Pi_{f_0}$
    at point $x_0.$

    Without loss of generality, we may assume that $f_0\in \bk$.
    Otherwise, we decompose $f_0$ as follows
    $$
    f_0=\sum_{i=1}^m\a_if_0^{(i)},
    $$
    where $\a_i>0,$ $f_0^{(i)}\in \bk$. Since
    $f^{(i)}_0(P(x_0))\leq 0,$ it follows from
    $$
    f_0(P(x_0))=\sum_{i=1}^m\a_if_0^{(i)}(P(x_0))=0
    $$
    that $f^{(i)}_0(P(x_0))=0, \ i=\overline{1,m}$. We then take one
    of $f_0^{(i)}$ instead of $f_0$.

    Let $H_i=\{x: \varphi_i(x)=c_i\}, i=\overline{1,n-k}$ be  supporting planes to $\ce_n(Q,A,b)$ at $x_0$ and
    $\varphi_i, \ i=\overline{1,n-k}$ be linearly independent functionals. By similar arguments discussed above,
    we can conclude that for each hyperplane $H_i$ there exists $f_i\in
    \bk$ such that $f_i(P(x_0))=0$ and $H_i$ is a tangent hyperplane
    to $\Pi_{f_i}=\{x\in\br^n: f_i(P(x))\leq 0\}$ at the point $x_0$. The linear
    independency of the functionals $\{f_i\},$ $i=\overline{1,n-k}$ follows from the linear
    independency  of the functionals $\{\varphi_i\},$ $i=\overline{1,n-k}.$ Indeed, if
    \begin{equation*}
    f_1=\sum\limits_{i=2}^{n-k}\l_if_i,
    \end{equation*}
    then the equation of a tangent hyperplane to $\Pi_{f_1}$ at the point $x_0$ has the form
    \begin{equation*}
    \sum\limits_{i=2}^{n-k}\l_i\varphi_i(x)=\sum\limits_{i=2}^{n-k}\l_ic_i,
    \end{equation*}
    i.e., $\varphi_1=\sum\limits_{i=2}^{n-k}\l_i\varphi_i$ and this contradicts to the linear independency of the functionals $\{\varphi_i\},$ $i=\overline{1,n-k}.$ This completes the proof.
\end{proof}

\begin{corollary}\label{c1}
    For any $x_0\in \partial \ce_n(Q,A,b),$ there exists $f\in \bk$ such that
    $f(P(x_0))=0.$
\end{corollary}
The proof immediately follows from the fact that any boundary
point has an order at most $n-1.$

\begin{corollary}\label{1c2}
    Every vertex  of  the set $\ce_n(Q,A,b)$ is a solution of the elliptic operator equation \eqref{1}.
\end{corollary}
\begin{proof} If $x_0$ is a vertex then there are $n$  linear independent
    functionals $f_1,\cdots,f_n\in \bk$ for which $f_i(P(x_0))=0,$
    $i=1,\cdots,n.$ Hence, we obtain that $P(x_0)=0.$ This completes the proof.
\end{proof}

We have shown in the previous section that the ellipticity of the elliptic operator equation \eqref{1} is stable under
small perturbations of the quadratic elliptic operator $Q$ as well as so do the conditions ${\overline K}'_Q\cap (-{\overline K}'_Q)=\{0\}, \ {\rm int}\ce_n(Q,A,b)\neq\emptyset$.

\begin{definition}\label{d}
    A solution $x_0$ of the equation
    $$P(x)\equiv Q(x)+Ax+b=0$$
    is called {\it stable} if for any $\e>0$ there exists $\d>0$ such that the
    perturbed  equation
    $${\tilde P}(x)\equiv{\tilde Q}(x)+{\tilde A}x+{\tilde b}=0$$
    has a solution ${\tilde x}_0$ such that
    $\|{\tilde x}_0-x\|<\e$ whenever $\|{\tilde Q}-Q\|<\d, \|{\tilde A}-A\|<\d$
    and $\|{\tilde b}-b\|<\d.$
\end{definition}

\begin{theorem}\label{t6}
    A solution of the elliptic operator equation \eqref{1} is stable if and only if it is
    a vertex of the set $\ce_n(Q,A,b)$.
\end{theorem}

\begin{proof} \texttt{If part.}  Let $x_0$ be a vertex of the set $\ce_n(Q,A,b).$ Then due to Corollary \ref{1c2}, it is
    a solution of the elliptic operator equation \eqref{1}. In order to prove its stability,  it is
    enough to show that the mapping $P:\br^n\to\br^n$ which is given by
    \begin{equation}\label{235}
    P(x)\equiv Q(x)+Ax+b
    \end{equation}
    is a local homeomorphism on
    some neighborhood of the point $x_0.$

    Let $P'(x_0)$ be the Frechet derivative of the mapping $P$ at the
    point $x^0$.  Now we will prove that there exists $[P'(x_0)]^{-1}$.
    We assume the contrary, i.e. it  does not exist.  We then have that
    $P'(x_0)h=0$ for some $h\ne 0$. Hence for any  $ f\in {\overline
        K}'_Q$ and $\l\in \br$ one has that
    \begin{equation*}
    f(P(x_0+\l h))=f(P(x_0))+\l
    f(P'(x_0)h)+\frac{1}{2}\l^2f(P''(\xi)(h,h))
    \end{equation*}
    where $P''(\xi)$ is the second  derivative of the mapping $P$ at the point $\xi\in [x_0,
    x_0+\l h].$

    The equalities $P(x_0)=0$, $P'(x_0)h=0$ and $P''(\xi)(h,h)=2Q(h)$
    imply that
    \begin{equation}\label{124}
    f(P(x_0+\l h))=\l^2f(Q(h))\geq 0.
    \end{equation}
    Consequently, the straight line $\{x_0+\l h\}_{\l\in\br}$ is a subset of
    the tangent hyperplane to $\Pi_f$ at $x_0$ for every $f\in \bk.$
    As we showed in the proof of Theorem \ref{t5}that any
    tangent hyperplane to $\ce_n(Q,A,b)$ at $x_0$ contains the straight line
    $\{x_0+\l h\}_{\l\in \br}$. So, the boundary point $x_0$ having
    an order more than 1 can not be a vertex.   Hence, there does exist  $[P'(x_0)]^{-1}$. Consequently $P:\br^n\to\br^n$ is
    a local homeomorphism on some neighborhood of $x_0,$ i.e. $x_0$ is a
    stable solution.

    \texttt{Only if part.} Let $x_0$ be a stable solution. We assume that
    $x_0$ is not a vertex of the set $\ce_n(Q,A,b).$ Then intersections of all
    supporting  hyperplanes to $\ce_n(Q,A,b)$ at the point $x_0$ contains at least
    one line, say  $\{x_0+\l h\}_{\l\in \br}, \ h\ne 0$. It is clear that $\{x_0+\l h\}_{\l\in \br}$ must be a tangent line
    to each ellipsoid $\ce_f$ at the point $x_0$. Thus, we have that $f(P(x_0+\l
    h))\geq 0, \ \forall \l\in\br, \ \forall f\in K'_Q$. The equality
    $f(P(x_0+\l h))=\l f(P'(x_0)h)+\l^2 f(Q(h))$ yields that $\l
    f(P'(x_0)h)+\l^2f(Q(x))\geq 0$ for any $\l\in \br$ and every
    $f\in K'_Q.$ One can then see that $f(P'(x_0)h)=0$ for all $f\in
    K'_Q.$ Since $K'_Q$ is a  solid cone, we obtain that $P'(x_0)h=0, \ h\ne 0.$ Therefore, $P'(x_0)$ does not have
 inverse, which contradicts to the stability of the solution $x_0.$ This completes the proof.
\end{proof}

\begin{remark} It is clear that each stable solution of the elliptic operator equation \eqref{1} is isolated and
an intersection of boundaries of $n$ ellipsoids in $\br^n$ can have at most $2^n$ isolated points. Thus the set of solutions $X_n(Q,A,b)$ contains at most $2^n$ stable solutions.
\end{remark}

Let us recall several notions from the vector field theory
which would be needed for our further study. Let $\O$ be a bounded domain in $\br^n$ and $\partial\O$ be its
boundary. We assume that a continuous regular vector field $\Phi(x)$ is given
on $\partial\O,$ i.e., $\Phi(x)$ is a nonzero vector filed for any $x\in \partial\O.$ We define a
continuous mapping $\frac{\Phi(x)}{\|\Phi(x)\|}: \partial\O\to S^{n-1}$ where $S^{n-1}=\{x\in \br^n:
\|x\|=1\}$ is a unit sphere. The degree of this mapping is called {\it an index} of the vector
field $\Phi$ on $\partial\O$ and it is denoted by $\g(\Phi, \partial\O).$ Note that
the index of the vector field $\Phi$ is an integer number. In particular, if the image of $\partial\O$ under the mapping ${\Phi(\cdot)\over \|\Phi(\cdot)\|}$ does not cover the
unit sphere $S^{n-1}$ then $\g(\Phi,\partial\O)=0.$ It is well known that if a number of zeros of the vector field $\Phi$
inside of $\O$ is finite and $\g(\Phi, \partial\O)=0$ then that number must be even.

\begin{theorem}\label{t7}
    An elliptic operator equation \eqref{1} has an even  (possibly, zero) number of stable solutions.
\end{theorem}

\begin{proof} Let $f\in K'_Q.$  Since an ellipsoid
    $\ce_f=\{x\in\br^n: f(P(x))\leq 0\}$  is bounded and $X_n(Q,A,b)\subset \ce_f,$ there exists $r>0$ such that $f(P(x))>0$ for all $x$ with
    $\|x\|=r.$ We denote by $\O=\{x: \|x\|\leq r\}$. Then we have that $ \partial\O=\{x: \|x\|=r\}$.  Therefore,
we get for $x\in \partial\O$ that
    $$f\bigg(\frac{P(x)}{\|P(x)\|}\bigg)>0.$$
    Hence, the image of $\partial\O$ under the mapping
    $\frac{P(\cdot)}{\|P(\cdot)\|}$ does not coincide with the unit
    sphere $S^{n-1}$. Therefore, $\g(P, \partial\O)=0.$ Consequently, since stable solutions of the elliptic operator equation \eqref{1} are finite, it must be even. This completes the proof.
\end{proof}

\subsection{Some examples}


Now we are going to provide some examples.

\begin{example}\label{a}
    Let us consider the following equation in $\br^2$
    \begin{equation}\label{125}
    \left\{ \begin{array}{ll}
    x^2_1=x_1,\\[2mm]
    x_2^2=x_2.\\
    \end{array}
    \right.
    \end{equation}
    It is clear that $Q(x)=(x_1^2, x_2^2)$, $Ax=(-x_1,-x_2),$ $b=0.$
    One can find that
    \begin{eqnarray*}
        K'_Q=\{(\xi_1,\xi_2): \ \ \xi_1>0, \quad \xi_2>0\}, \quad \bk=\{(\a, 0), (0,\b): \ \ \a>0,\quad \b>0\},\\
        \Pi_1=\{(x_1,x_2): x_1^2\leq x_1\},\quad  \Pi_2=\{(x_1,x_2): x^2_2\leq x_2\}, \quad \ce_n(Q,A,b)=\Pi_1\cap\Pi_2.
    \end{eqnarray*}
    In this instance, $\ce_n(Q,A,b)$ is a
    unit square with vertices $(0,0), (0,1), (1,0), (1,1)$ which are
    only stable solutions of the elliptic operator equation \eqref{125}.

\end{example}

\begin{example}\label{b}
    Let us consider the following equation in $\br^2$
    \begin{equation}\label{126}
    \left\{ \begin{array}{ll}
    x^2_1=x_2,\\[2mm]
    x_2^2=x_1. \\
    \end{array}
    \right.
    \end{equation}
    In this case, we have that  $Q(x)=(x_1^2, x_2^2),$ $Ax=(-x_2,-x_1),$ $b=0.$ Similarly, one can find that
    \begin{eqnarray*}
        K'_Q=\{(\xi_1,\xi_2): \ \ \xi_1>0, \quad \xi_2>0\}, \quad \bk=\{(\a, 0), (0,\b): \ \ \a>0,\quad \b>0\},\\
        \Pi_1=\{(x_1,x_2): x_1^2\leq x_2\}, \quad  \Pi_2=\{(x_1,x_2): x^2_2\leq x_1\}, \quad \ce_n(Q,A,b)=\Pi_1\cap\Pi_2.
    \end{eqnarray*}
    One can see that $\ce_n(Q,A,b)$ has only vertices
    $(0,0), (1,1)$ which are stable solutions of the elliptic operator equation \eqref{126}.

\end{example}

In general, all extreme points of the set $\ce_n(Q,A,b)$ are not necessary to be solutions of the elliptic operator equation \eqref{1}.

\begin{example}\label{c} Let $X=(x_{ij})_{i,j=1}^n$ is a symmetric matrix of an order
    $n\times n.$ We consider the following equation in $\br^{\frac{n(n+1)}{2}}.$
    \begin{equation}\label{127}
    X^2=X.
    \end{equation}
The solution of the equation \eqref{127} is an orthogonal projector in the space
    $\br^n.$ As a linear functional in $\br^{\frac{n(n+1)}{2}},$ we
    consider $f=(\xi_{ij})_{i,j=1}^n$ such that $\xi_{ij}=\xi_{ji}.$
    Consequently, $f(X)$ is defined by
    $$f(X)=\sum_{i,j=1}^n\xi_{ij}x_{ij}.$$
    In particular, if
    $$f_0 =
    \left(
    \begin{array}{cccc}
    1 & 0 & \cdots & 0\\[2mm]
    0 & 1 & \cdots & 0\\[2mm]
    \cdot & \cdot & \cdots & \cdot\\[2mm]
    0 & 0 & \cdots & 1\\
    \end{array}
    \right),
    $$
    then $f_0(X)=tr(X)$ is the trace of the matrix $X$. For this linear functional $f_0,$ we have that
    $$ f_0(Q(X))=f_0(X^2)=\sum_{i,j=1}^nx_{ij}^2\geq \|X\|^2,$$
    where $\|X\|$ is a norm of the matrix $X$ associated with the
    Euclidean norm in $\br^n.$

    Therefore, the equation \eqref{127} is an elliptic operator equation.
One can find that
    \begin{eqnarray*}
        && K'_Q=\{f: f=(\xi_{ij})_{i,j=1}^n-{\rm \ a \ positive \ defined
            \ symmetric \ matrix}\},\\
        && \ce_n(Q,A,b)=\bigcap_{f\in K'_Q}\ce_f=\{X: X-X^2 -{\rm a \ nonnegative  \ defined \ symmetric \ matrix}\}.
    \end{eqnarray*}

Due to Theorem \ref{t2}, the set of solutions $X_n(Q,A,b)$ of the elliptic operator equation \eqref{c} (in other words, the set of orthogonal projectors of the space $\br^n$)  is  a subset of the set ${\textbf{\textup{Extr}}}\ce_n(Q,A,b)$ and a convex independent in the following sense if
    $X$ is a solution then it cannot be represented by
    $$ X=\sum^k_{i=1}\a_i X^{(i)},$$
    where $X^{(i)}\in X_n(Q,A,b)$ are solutions and $\a_i\geq 0, \ \sum^k_{i=1}\a_i=1, \ k\geq 2$ such that at least two coefficients $\a_1,\dots,\a_k$ are nonzero.

    We want to prove that $X_n(Q,A,b) = \textbf{\textup{Extr}}\ce_n(Q,A,b).$ Let $T$ be a real orthogonal matrix in
    $\br^n,$ i.e. $T^{-1}=T^{t}.$ Let us consider a map $\mathcal{T}:X\to T^tXT.$ It is obvious that
    \begin{enumerate}
        \item[(i)] If $X\geq 0$ then $T^tXT\geq 0;$

        \item[(ii)] If $X$ is an orthogonal projector then $T^tXT$ is also
        an orthogonal projector;

        \item[(iii)] If $X=\a Y+\b Z$ then $T^tXT=\a T^tYT+\b T^tZT.$
    \end{enumerate}

    This means that $\mathcal{T}(X)=T^tXT$ is a convex mapping which maps $\ce_n(Q,A,b)$ into
    itself.  We assume that  $X_0\in \textbf{\textup{Extr}}\ce_n(Q,A,b)$ is not an orthogonal
    projector. Let us reduce $X_0$ to a diagonal form:
    $$T^tX_0T={\rm diag}(\l_1,..,\l_n).$$
    Since $T^tX_0T\in \ce_n(Q,A,b),$ we get that $(T^tX_0T)^2\leq T^tX_0T.$ Consequently, we obtain that
    $\l^2_1\leq \l_1,\cdots, \l_n^2\leq \l_n$ and $\l_1,\cdots,\l_n\in
    [0,1].$ Since $X_0$ is not an orthogonal projection, $T^tX_0T$ is not an orthogonal
    projector. Thus, at least one of numbers $\l_i,$ $i=1,\dots,n$, say $\l_1,$ is
    different from $0$ and $1.$ Then
    \begin{equation}\label{128}
    T^tX_0T=\l_1Y+(1-\l_1)Z,
    \end{equation}
    where $Y={\rm diag}(1, \l_1,\dots,\l_n), \ Z={\rm diag}(0,
    \l_2,\dots,\l_n),$ with $Y,Z\in \ce_n(Q,A,b).$  Since $\mathcal{T}$ and ${\mathcal{T}}^{-1}$ are linear mapping, it follows from \eqref{128} that
    \begin{eqnarray*}
        X_0={\mathcal{T}}^{-1}(T^tX_0T)={\mathcal{T}}^{-1}(\l_1Y+(1-\l_1)Z)=\l_1TYT^t+(1-\l_1)TZT^t
    \end{eqnarray*} with $TYT^t\in \ce_n(Q,A,b),$ $TZT^t\in \ce_n(Q,A,b).$ This contradicts to $X_0\in \textbf{\textup{Extr}} \ce_n(Q,A,b).$ Hence, we have that $X_n(Q,A,b)={\textbf{\textup{Extr}}}\ce_n(Q,A,b).$

    One can show that $\ce_n(Q,A,b)$ has only two vertices $\Theta={\rm diag}(0,\dots,0)$ and $I={\rm
        diag}(1,\dots,1).$ Consequently, the equation \eqref{127} has two
    stable solutions. All other solutions are not stable.

    For instance, in the case $n=2,$ the set $X_2(Q,A,b)$ has three connected components: $\{\Theta\}, \ \{I\}$
    and one dimensional projectors
    $$
\left(
    \begin{array}{cc}
    \a & \pm \sqrt{\a(1-\a)}\\[2mm]
    \pm \sqrt{\a(1-\a)} & 1-\a\\
    \end{array}
    \right), \ \a\in[0,1].
    $$
In the general case, $X_n(Q,A,b)$ has $n+1$ connected components which are
    null dimensional, one dimensional, two  dimensional, and so on $n-$dimensional projectors.
\end{example}

\section{An iterative method for stable solutions}


We are going to show an existence of stable solutions under the mild condition by means of the Newton-Kantorovich method.

\subsection{The Newton-Kantorovich method}


\begin{proposition}
    The set
    \begin{equation}\label{28}
        \bd=\br^n\setminus \bigg(\bigcup_{f\in \bk}\Pi_f\bigg)
    \end{equation}
    is an open set.
\end{proposition}
\begin{proof}
    We will show that $\bigcup\limits_{f\in\bk}\Pi_f$ is a closed set. Let
    $\{x_n\}\in \bigcup\limits_{f\in \bk}\Pi_f$ such that  $x_n\to x_0.$ Then
    there exists $f_n\in \bk,$ $\|f_n\|=1$ such that $f_n(P(x_n))\leq 0.$
    Since  $\bk\cap \bb(0,1)$ is a compact set, without loss any generality, we may assume that
$f_m\to f_0,\   f_0\in \bk. $   Then It follows from $x_n\to x_0,$ $f_n\to f_0,$ and $f_n(P(x_n))\leq 0$ that $f_0(P(x_0))\leq 0.$ Since $f_0\in \bk,$ one has that $x_0\in
    \bigcup\limits_{f\in\bk}\Pi_f.$ Consequently, $\bigcup\limits_{f\in\bk}\Pi_f$ is a
    closed set and $\bd$ is an open set. This completes the proof.
\end{proof}

\begin{theorem}\label{t8} Let $\bd_0$ be a connected component of $\bd$ and ${\overline
        \bd}_0$ be its closure. If there exists $x_0\in \bd_0$ such that ${\overline \bd}_0$ does
    not contain any straight line passing through $x_0$ then there
    exists a stable solution of the equation \eqref{1} belonging to the set
    ${\overline \bd}_0.$
\end{theorem}
\begin{proof} We prove this theorem in a few steps.

    {\textsf{1 Step.}}

    Let us prove that the inverse $[P'(x_0)]^{-1}$ of the mapping $P'(x_0)$ exists. We assume the contrary i.e. the inverse $[P'(x_0)]^{-1}$ does not exist. We then have that $P'(x_0)h=0$ for some $0\ne h\in \br^n.$
    From which for any  $ f\in {\overline K}'_Q$ and $\l\in \br$, one
    can find that
    \begin{equation*}
        f(P(x_0+\l h))=f(P(x_0))+\l
        f(P'(x_0)h)+\l^2f(Q(h))=f(P(x_0))+\l^2 f(Q(h)).
    \end{equation*}
Since $x_0\in \bd_0$ and $ f\in {\overline K}'_Q,$  we obtain that $f(P(x_0))>0$, $f(Q(h))\geq 0.$
    Therefore, $f(P(x_0+\l h))>0$ for any $\l\in\br.$ This means that $\bd_0$ contains a straight line
    $\{x_0+\l h\}_{\l\in \br}$ which contradicts to the condition of theorem. Consequently, there exists $[P'(x_0)]^{-1}.$

    {\textsf{2 Step.}}

    We setup $x_\l=x_0-\l [P'(x_0)]^{-1}(P(x_0))$ where $0\leq \l\leq 1.$ We are going to show that $x_\l\in \bd_0$ for $0\leq \l<1$ and $x_1\in {\overline \bd}_0.$ Indeed, by expanding $f(P(x_\l))$ in the Taylor series at $x_0,$ we obtain that
    \begin{eqnarray*}
        f(P(x_\l)&=&f(P(x_0))-\l f(P'(x_0)[P'(x_0)]^{-1}(P(x_0)))+\l^2 f(Q(h_0))\\
        &=&(1-\l)f(P(x_0))+\l^2 f(Q(h_0)),
    \end{eqnarray*}
    where $h_0=[P'(x_0)]^{-1}(P(x_0))$ and $f\in {\overline K}'_Q.$
    It follows from $f(P(x_\l))>0,$ $0\leq \l<1$ that $x_\l\in \bd_0.$ Consequently, we get that $x_1\in {\overline \bd}_0.$

    {\textsf{3 Step.}}

    We draw all possible tangent hyperplanes to
    $\Pi_f=\{x: f(P(x))\leq 0\},$ where $f\in \bk,$ passing through the point $x_0$   such
    that the tangent point belongs to ${\overline \bd}_0$. Let $H^+_f$ be a closed semi-space defined by a tangent hyperplane to
    $\Pi_f,$ containing $\Pi_f.$ It is obvious that an intersection of all hyperplanes $H^+_f$ contains $\ce_n(Q,A,b)=\bigcap\limits_{f\bk}\Pi_f.$
    Since ${\rm{\textbf int}}\ce_n(Q,A,b)\ne \emptyset,$ the set $\bigcap\limits_{f\in \bk}H^+_f$ is
    a nonempty  closed solid convex set. We define a set $C_0=x_0-\bigcap\limits_{f\in \bk}H^+_f.$ We are going to show that $C_0$ is a cone. Indeed, $x_0-C_0=\bigcap\limits_{f\in \bk}H^+_f$ and $x_0+C_0\subset{\overline \bd}_0.$ Since ${\overline \bd}_0$ does not contain any straight line passing through the point $x_0,$ the set $C_0$ is a cone.

    {\textsf{4 Step.}}

    Let us now prove that $(x_0-C_0)\cap {\overline \bd}_0$ is a
    bounded set. Let $h\in C_0$ with $\|h\|=1.$ Since the straight line $x_0+\l h,$ $\l\in\br$ is contained in ${\overline
        \bd}_0,$ we then get that
    $$f(P(x_0+\l h))\geq
    0,$$ for all $\l\geq 0,$ $f\in \bk$ and there exists $f_0\in
    {\overline K}'_Q$ and $\l_0<0$ such that
    \begin{equation}\label{129}
        f_0(P(x_0+\l_0 h))<0.
    \end{equation}
     The inequality \eqref{129} means that $(x_0-C_0)\cap {\overline
        \bd}_0$ is bounded in the direction $-h$ from the point $x_0.$ Since $P, f_0$ are continuous mapping, the inequality \eqref{129} is satisfied in a sufficiently small neighborhood of $h,$ namely, it follows from $\|h-h'\|<\e, \|h'\|=1$ that
    $f_0(P(x_0+\l_0 h'))<0.$

    We know that $\{h: h\in C_0, \|h\|=1\}$ is compact set. Then
    there exist $h_1,\cdots,h_l\in C_0,$ $\|h_i\|=1,$
    functionals $f_1,\cdots,f_l\in K'$ and numbers $\l_1,\cdots,\l_l<0$ such that
    \begin{equation}\label{130}
        f_i(P(x_0+\l_i h_i))<0, i=\overline{1,l}.
    \end{equation}
    Moreover, for any $h\in C_0$ with $\|h\|=1$ one can choose $i$
    such that
    \begin{equation}\label{131}
        f_i(P(x_0+\l_i h))<0.
    \end{equation}
    Consequently $(x_0-C_0)\cap {\overline \bd}_0$ is bounded within the ball $\bb(\Theta,\max\limits_{1\leq i\leq l}|\l_i|)$
    in an any direction from $x_0.$

    {\textsf{5 Step.}}

    It is clear from the previous steps that  $x_1=x_0-[P'(x_0)]^{-1}P(x_0)\in
    x_0-C_0$ and the point $x_1$ satisfies the condition of Theorem. If we can construct a solid cone $C_1$ corresponding to the point $x_1$ by the similar way presented in the previous steps. We then have that $C_1\subset C_0.$ By continuing this process one can construct the Newton-Kantorovich iterations as follows
    \begin{equation}\label{132}
        x_{k+1}=x_k-[P'(x_k)]^{-1}P(x_k), \ \ k=0,1,\dots
    \end{equation}
    It is clear that this is a decreasing sequence along the solid cone $C_0$.
    Since $\{x_k\}\subset x_0-C_0,$ $k=0,1,\cdots$ then $x_k$ is bounded. It is known that in finite dimensional space every solid cone is regular, in other words, a monotone bounded sequence converges with respect to norm. Therefore  $\|x_k-x^*\|\to 0,$ as $k\to \infty$ for some $x^*.$ Finally, $x^*\in x_0-C_0,$ and
    $[P'(x^*)]^{-1}$ exists. Consequently, $x^*$ is a stable solution of equation $P(x)=0.$
\end{proof}

\begin{remark} Let $\Phi(x)=x-[P'(x)]^{-1}P(x).$ As we already showed that if $\Phi(x)$ is well defined in $\overline{\bd}$ then $\Phi(x)\in {\overline \bd}.$
\end{remark}

\subsection{Some examples}


We shall illustrate the Newton-Kantorovich method in several examples.
\begin{example}\label{i}
    Let us consider the following elliptic operator equation in $\br^2.$
    \begin{equation*}
        \left\{ \begin{array}{ll}
            \xi^2_1=1\\[2mm]
            \xi_2^2=1,\\
        \end{array}
        \right.
    \end{equation*}
    where $x=(\xi_1,\xi_2)\in\br^2.$ Then $\bd=\{x\in\mathbb{R}^2: |\xi_1|>1, \ |\xi_2|>1\}$ has four connected components $\bd_1=\{x\in\mathbb{R}^2: \xi_1>1,\ \xi_2>1\},$ $\bd_2=\{x\in\mathbb{R}^2: \xi_1<-1, \ \xi_2>1\},$ $\bd_3=\{x\in\mathbb{R}^2: \xi_1<-1, \ \xi_2<-1\},$ and $\bd_4=\{x\in\mathbb{R}^2: \xi_1>1, \ \xi_2<-1\}.$  Any initial point $x_0\in \bd$ satisfies the condition of Theorem \ref{t8}. Hence, we have that
    \begin{equation*}
        x^*=\left\{\begin{array}{llll}
            (1,1), &  x_0\in \bd_1 \\[2mm]
            (-1,1),&  x_0\in \bd_2 \\[2mm]
            (-1,-1),& x_0\in \bd_3 \\[2mm]
            (1,-1), & x_0\in \bd_4 \\
        \end{array}\right.
    \end{equation*}
Consequently, if $x_0\in\bd$ then the Newton-Kantorovich iteration
    converges to one of the stable solutions.
\end{example}

\begin{example}\label{ii}
    Let us consider the following equation in $\br^2.$
    \begin{equation*}
        \left\{ \begin{array}{ll}
            \xi^2_1=\xi_2\\[2mm]
            \xi_2^2=\xi_1\\
        \end{array}
        \right.
    \end{equation*}
    where $x=(\xi_1,\xi_2)\in\br^2.$ Then $\bd=\{x\in\mathbb{R}^2: \xi_1<\xi_2^2, \ \xi_2<\xi_1^2\}$ has two connected components $\bd_1=\{x\in\mathbb{R}_{+}^2:  \xi_1<\xi_2^2, \ \xi_2<\xi_1^2\}$ and $\bd_2=\bd\setminus \bd_1.$ Any initial point $x_0\in \bd_1$ satisfies the condition of Theorem \ref{t8} and the Newton-Kantorovich iteration converges to $x^{*}=(1,1).$

    Let $\ell$ be the common tangent line to parabolas
    $\Pi_1$ and $\Pi_2.$ It is obvious that
    $\ell=\{(\xi_1,\xi_2):\xi_1+\xi_2=-1\}$ and if
    $x_\circ=(\xi_1^\circ,\xi_2^\circ)\in \bd_2$ with $\xi_1^\circ+\xi_2^\circ>-1$ then
    the condition of Theorem \ref{t8} is satisfied. Therefore, the Newton-Kantorovich
    iteration starting from this point converges to the stable point
    $x^*=(0,0).$ For all the rest points $x\in\bd_2,$ the condition
    of Theorem \ref{t8} is not satisfied.
\end{example}

\begin{example}\label{iii}
    Let us consider the following equation in $\br^2.$
    \begin{equation*}
        \left\{ \begin{array}{ll}
            \xi^2_1=\xi_2\\[2mm]
            \xi_2^2=4\xi_2-3\\
        \end{array}
        \right.
    \end{equation*}
    where $x=(\xi_1,\xi_2)\in\br^2.$ Then $\bd=\{x\in\mathbb{R}^2: \xi_2<\xi_1^2, \  \xi_2^2>4\xi_2-3\}$ has three  connected components $\bd_1=\{x\in\mathbb{R}^2: \xi_2<\xi_1^2, \ \xi_1>\sqrt{3}, \ \xi_2>3\},$ $\bd_2=\{x\in\mathbb{R}^2: \xi_2<\xi_1^2, \ \xi_1<-\sqrt{3}, \ \xi_2>3\},$ and $\bd_3=\{x\in\mathbb{R}^2: \xi_2<\xi_1^2, \ \xi_2<1\}.$ It follows from Theorem \ref{t8} that the Newton-Kantorovich iteration starting from any point $x_0\in \bd_1($resp. $\bd_2)$ converges to a stable solution $(\sqrt{3}, 3)($resp. $(-\sqrt{3}, 3)$$)$.

Let $\bd_3^{(1)}=\{x\in\mathbb{R}^2: \xi_2<\xi_1^2, \ \xi_1>0, \ 0<\xi_2<1\}$ and $\bd_3^{(2)}=\{x\in\mathbb{R}^2: \xi_2<\xi_1^2, \ \xi_1<0, \ 0<\xi_2<1\}.$    These subsets of the set $\bd_3$ have the following property: for any $x_0\in
    \bd_3^{(1)}($resp. $\bd_3^{(2)})$ there is no a straight line passing through
    $x_0$ and lying inside $\bd_3.$ Then for any $x_0\in \bd_3^{(1)}($resp. $\bd_3^{(2)})$
    the Newton-Kantorovich iteration converges to a stable solution $(1,1)($resp. $(-1,1))$. Simple calculations show that for any other points
    of $\bd_3$ the Newton-Kantorovich  iteration does not converge. In
    this example, the set ${\overline \bd}_1$ as well as ${\overline \bd}_2$ contains
    one stable solution and the set ${\overline \bd}_3$ contains two
    stable solutions.
\end{example}

\begin{example}\label{iv}
    Let us consider the following equation in $\br^2.$
    \begin{equation*}
        \left\{ \begin{array}{ll}
            \xi^2_1=\xi_2\\[2mm]
            \xi_2^2=1\\
        \end{array}
        \right.
    \end{equation*}
    where $x=(\xi_1,\xi_2)\in\br^2.$ Then $\bd=\{x\in\mathbb{R}^2: \xi_2<\xi_1^2, \  |\xi_2|>1\}$ has three  connected components $\bd_1=\{x\in\mathbb{R}^2: \xi_2<\xi_1^2, \ \xi_1>1, \  \xi_2>1\},$ $\bd_2=\{x\in\mathbb{R}^2: \xi_2<\xi_1^2, \ \xi_1<-1, \  \xi_2>1\},$ and  $\bd_2=\{x\in\mathbb{R}^2: \xi_2<-1\}.$ It follows from Theorem \ref{t8} that for any initial point taken from $\bd_1($resp. $\bd_2)$,  the Newton-Kantorovich iteration
    converges to the stable solution $(1,1)($resp. $(-1,1))$. For any initial point taken from $\bd_3$, the Newton-Kantorovich iteration does not converge.
\end{example}


\section{The rank of elliptic quadratic operators}

In this section we are going to introduce the concept of \textit{a rank} of the elliptic quadratic operator. By means of this concept, we are going to describe the cone $K'_Q.$

Let $Q:\br^n\to\br^n$ be an elliptic quadratic operator. Recall
that  $K'_Q$ denotes the set of linear continuous functionals $f:\br^n\to\br$
such that $f(Q(x))$ is a positive defined quadratic form. In the
sequel, we study a closed cone ${\overline K}'_Q.$   Let $\bk$
be the set of extremal rays of ${\overline K}'_Q.$ Due to
Krein-Milman theorem, we have that $\overline{co\bk}={\overline K}'_Q$ where $co \bk$ is the convex
hull of $\bk.$
Let ${\rm rg}_fQ$ stand for the rank of the quadratic
form $f(Q(x))$. It is clear that the rank ${\rm rg}_fQ$ of the quadratic form  $f(Q(x))$ is equal to the rank of the associated symmetric matrix $A$. Due to the construction of the set $ K'_Q$, one has that ${\rm rg}_fQ=n$ whenever $f\in K'_Q$ and ${\rm rg}_fQ<n$ whenever $f\in
\partial K'_Q$.

\begin{definition} The
    number
    $${\rm rg} Q=\max_{f\in \bk}{\rm rg}_fQ$$ is called a {\it rank} of the elliptic quadratic operator $Q:\br^n\to\br^n$.
\end{definition}

It is clear  that $1\leq {\rm rg} Q\leq n-1$ for any elliptic quadratic operator.

If $B$ is a linear invertible operator then quadratic operators
$Q(x)$ and $Q_{B}(x)=Q(Bx)$ have the same cone $K'$ since the
quadratic forms $f(Q(x))$ and $f(Q(Bx))$ both are simultaneously either positive defined or not. Let us consider  quadratic operators $Q_1$ and $Q_2$ such that $Q_2(x)=A Q_1(x)$ where $A$ is a linear invertible operator. Let
$K'_i,$ be the corresponding cone to $Q_i,$ $i=1,2.$ Then it follows from
$f(AQ_1(x))=(A^tf)(Q_1(x))$ (where $A^t$ is a transpose operator) that
$$
K_2'=(A^t)^{-1}K'_1.
$$

Let ${\rm \textbf{Isom}}\br^n$ be the set of all
isometries of $\br^n$. If $A,B\in {\rm \textbf{Isom}}\br^n$  and $Q\in\mathfrak{EQ_n}$ is
an elliptic quadratic operator then $AQ(B(\cdot))$ also is an elliptic
quadratic operator with
$${\rm rg}(AQ(B))={\rm rg Q}.$$

\begin{definition} An elliptic quadratic operator
    $Q:\br^n\to\br^n$ is called {\it homogeneous}
    of rank $r$ if one has that ${\rm rg}_fQ=r$ for any $f\in \bk$.
\end{definition}

\begin{example} An elliptic quadratic operator $Q(x)=(x^2_1+x^2_2,
    x^2_2+x^ 2_3, 2x_1x_3)$ is homogeneous of the rank 2 in $\br^3$. Indeed, if
    $f=(\l_1,\l_2,\l_3)$ then $K'_Q=\{f: \l_1>0, \l_2>0,
    \l^2_3<\l_1\cdot\l_2\}.$ Consequently, we get that $\bk=\{f: \l_1\geq 0, \l_2\geq
    0, \l_3=\pm\sqrt{\l_1\cdot\l_2}\}$ and
    \begin{eqnarray*}
        f(Q(x)) &= &\l_1(x^2_1+x^2_2) + \l_2(x^2_2+x^2_3)\pm 2\sqrt{\l_1\l_2}x_1x_2, \nonumber \\
        &= & (\sqrt{\l_1}x_1\pm \sqrt{\l_2}x_3)^2+ (\sqrt{\l_1}x_2\pm
        \sqrt{\l_2}x_3)^2,
    \end{eqnarray*}
for any $f\in \bk,$ i.e. ${\rm rg}_fQ=2$ for any $f\in \bk.$ This means this elliptic quadratic operator is homogeneous of the rank 2.
\end{example}

\begin{example} An elliptic quadratic operator $Q(x)=(x^2_1+x^2_2,
    x^2_2+x^2_3, 2x_1x_2)$ is not homogeneous of the rank 2 in $\br^3$.
    Indeed, one has that
    \begin{eqnarray*}
        &&K'_Q=\{f: \l_1>0, \l_2>0, \l^2_3<\l_1(\l_1+\l_2)\}, \\
        && \bk=\bigg\{f: f\ne 0, \l_1\geq 0, \l_2\geq 0, \l_3=\pm
        \sqrt{\l_1(\l_1+\l_2)}\bigg\},
    \end{eqnarray*}
    and $f(Q(x))=(x_1^2+x_2^2)^2$ for $f=(1,0,1)\in \bk$. This
    means that ${\rm rg}_fQ=1.$
\end{example}


\subsection{The rank of $k$}

Now we are going to describe the cone ${\overline
    K}'_Q$ for a homogeneous elliptic quadratic operators of rank $k$.

\begin{theorem}\label{t9} If ${\rm rg}Q=1$ then there are $A, B\in {\textbf{\textup{Isom}}}\br^n$ such that $AQ(Bx)=(x_1^2,x^2_2,\dots,x_n^2)$.
    Moreover, ${\overline K}'_Q$ is a miniedral cone i.e. $\bk$ contains
    exactly $n$ extremal rays.
\end{theorem}

\begin{proof} Since ${\overline K}'_Q$ is a solid cone, it has at least
    $n$ extremal rays. Let $f_1,\cdots,f_n$ be these extremal rays.
    Without loss of generality, we may assume that $f_1,\cdots, f_n$ are
    linearly independent. We define a linear operator $A:\br^n\to\br^n$
    as follows
    $$Ax:=(f_1(x),f_2(x),\cdots,f_n(x)).$$
The linearly independence of $f_1,\cdots,f_n$ implies that  $A\in
    {\rm \textbf{Isom}}\br^n.$ Then
    $$AQ(x)=(f_1(Q(x)),f_2(Q(x)),\cdots,f_n(Q(x))).$$
Since ${\rm rg}Q=1$ and $f_1,\cdots,f_n\in\bk,$ one gets that
    $${\rm rg}_{f_1}Q=\cdots={\rm rg}_{f_n}Q=1,$$
    i.e.
    $$
    f_1(Q(x))=(\varphi_1(x))^2 ,\cdots, f_n(Q(x))=(\varphi_n(x))^2,$$
    where $\varphi_1,\cdots,\varphi_n$ are some linear functionals. Hence,
    $$AQ(x)=((\varphi_1(x))^2,\cdots,(\varphi_n(x))^2).$$
Since $Q$ is the elliptic quadratic operator, the linear functionals
    $\varphi_1(x),\dots,\varphi_n(x)$ are linearly independent. If $B:\br^n\to\br^n$ is a linear operator which maps
    $(\varphi_1(x),\cdots,\varphi_n(x))$  to
    $(x_1,\cdots,x_n).$ Then
    \begin{eqnarray*}
        AQ(Bx)=(x^2_1,\dots,x_n^2),\quad {\overline K}'_1=\{f: \l_1\geq 0,\dots,\l_n\geq 0\},
    \end{eqnarray*}
    and $\bk_1$ exactly consists of $n$ extremal rays. Consequently,
    $\bk_1=(A')^{-1}\bk$ also contains exactly $n$ extremal rays. This completes the proof.
\end{proof}

\begin{theorem}\label{t10} If ${\rm rg}Q\geq 2$ then $\bk$ is an infinite set.
\end{theorem}

\begin{proof} Let ${\rm rg}Q= 2$ and $\bk$ be a finite set. Then ${\overline K}'_Q$
    is a polyhedral solid cone. We choose $f_1,\cdots,f_n\in \bk$ such
    that they are linearly independent and a convex hull of any $(n-1)$ of
    them does not intersect with $K'_Q.$ Since ${\rm
        rg}_{f_i}Q=2,$ we have that
    $$f_i(Q(x))=(\varphi_i(x))^2+(\psi_i(x))^2,\quad i=\overline{1,n},$$  where
    $\varphi_i $ and $\psi_i$ are some linear functionals. Without loss of generality, one can assume
    that $\varphi_1,\cdots,\varphi_n$ are linearly independent otherwise we may alternate some functionals
    $\varphi_i(x)$ with $\psi_i(x)$. We define linear operator $A,B$ as follows
    $$Ax:=(f_1(x),\cdots,f_n(x),\quad Bx:=(\varphi_1(x),\cdots,\varphi_n(x)).$$  So, we have that
    $$AQ(Bx)=(x^2_1+(\a_1(x))^2,\cdots,x_n^2+(\a_n(x))^2),$$
    where $\a_1(x),\cdots,\a_n(x)$ are some linear functionals.

By setting $f=(0,1,1,\dots,1),$ one can get
    $$f(AQ(Bx))=x^2_2+x^2_3+\cdots+x_n^2+(\a_2(x))^2+(\a_3(x))^2+\cdots+(\a_n(x))^2.$$
    Since $f\in co(f_2,\dots,f_n),$ the quadratic form $f(AQ(Bx))$
    can not be positive defined. Therefore, we obtain that ${\rm Ker} f(AQ(B))\ne
    \{0\}.$ Thus, the linear functionals $\a_2(x),\dots,\a_n(x)$ do not
    depend on $x_1.$ Analogously, if we let $f=(1,0,1,\dots, 1)$ then one can show that the
    linear functionals $\a_1(x),\a_3(x),\cdots,\a_n(x)$ do not depend on
    $x_2$.

    We continue this process. If we let $f=(1,\cdots, 1, \underbrace{0}_i, 1,\cdots, 1)$ then one can show that the linear functionals $$\a_1(x),\cdots,\a_{i-1}(x), \a_{i+1}(x),\cdots,\a_n(x)$$
    do not depend on $x_i.$ It follows from these arguments that
    $$\a_1(x)=\l_1x_1, \a_2(x)=\l_2x_2,\cdots, \a_n(x)=\l_nx_n.$$
    Hence
    \begin{equation*}
        AQ(Bx)=((1+\l_1^2)x_1^2, (1+\l_2^2)x_2^2,\dots,(1+\l_n^2)x_n^2),
    \end{equation*}
    and ${\rm rg}AQ(B)=1$ which contradicts to ${\rm rg}Q=2.$ By means of the mathematical induction methods one can prove the assertion of the theorem in the case ${\rm rg}Q= k>2.$ This completes the proof.
\end{proof}

\begin{corollary}
One has that ${\rm rg}Q=1$ if and only if ${\overline K}'_Q$ is a miniedral cone.
\end{corollary}

\begin{theorem}\label{t11} If ${\rm rg}Q=n-1$
    then $\bk=\partial {\overline K}'_Q$.
\end{theorem}

\begin{proof} We will prove that $\partial {\overline K}'_Q\subset\bk.$ Indeed, if $f\in \partial {\overline K}'_Q$ then it can be  represented as $f=\sum\limits^r_{i=1}f_i,$ where $f_i\in \bk$ and $r\leq
    n.$ Then we have that
    $$f(Q(x))=\sum^r_{i=1}f_i(Q(x)).$$

    Since $f(Q(x)), \ f_i(Q(x))$ are the nonnegative defined quadratic
    forms, we have that
    $${\rm rg}_fQ\geq \max_{1\leq i\leq r}{\rm rg}_{f_i}Q=n-1.$$
    On the other hand, since $f\in \partial {\overline K}'_Q,$ we get that ${\rm
        rg}_fQ<n.$ Hence
    \begin{equation}\label{133}
        {\rm rg}_fQ=n-1.
    \end{equation}
    for any $f\in \partial {\overline K}'_Q$. This means that $\partial {\overline K}'_Q\subset\bk.$

    We will prove that $\partial {\overline K}'_Q\supset\bk.$ We assume the contrary, i.e., one can choose  $f_1, f_2\in \bk$
    such that
    \begin{equation}\label{134}
        cone\{f_1, f_2\}\cap K'_Q=\emptyset,
    \end{equation}
    where $cone\{f_1, f_2\}$ is a conical hull of $f_1, f_2.$  A linear
    operator $B\in {\rm \textbf{Isom}}\br^n$ can be chosen such that
    \begin{eqnarray*}
        && f_1(Q(Bx))=x^2_1+x^2_2+\cdots+x^2_{n-1}, \\
        &&
        f_2(Q(Bx))=(\varphi_1(x))^2+(\varphi_2(x))^2+\cdots+(\varphi_{n-1}(x))^2.
    \end{eqnarray*}

    It follows from \eqref{134} that the linear functionals
    $\varphi_1(x),\dots,\varphi_{n-1}(x)$ do not depend on $x_n.$
    Since  $f_1(Q(Bx))$ and $f_2(Q(Bx))$ quadratic forms are positive defined with
    respect to $x_1,\dots,x_{n-1}$ they can
    be simultaneously represented in a diagonal form. Namely, there are $A, B_1\in
    {\textbf{Isom}}\br^n$ such that
    \begin{eqnarray*}
        &&f_1(AQ(B_1x))=x^2_1+x^2_2+\cdots.+x^2_{n-1},\\
        &&f_2(AQ(B_1x))=\l_1x_1^2+\l_2x_2^2+\cdots+\l_{n-1}x_{n-1}^2, \
        \l_i>0.
    \end{eqnarray*}
    Without loss of generality one can assume that $\l_i\leq 1$ and
    $\max_i\l_i=1.$ Then $(f_1-f_2)(AQ(B_1x))$ is a nonnegative
    defined quadratic form with rank $\leq n-2.$ Consequently, we obtain that
    $f_1-f_2\in
    \partial {\overline K}'_Q.$ Then, we have that
    $$ {\rm rg}_{f_1-f_2}(AQ(B_1x))\leq n-2$$
    and it contradicts to \eqref{133}. This completes the proof.
\end{proof}

Here is an example for a homogeneous elliptic quadratic operator of order $n-1.$
\begin{example} Let us consider the following elliptic quadratic operator
    \begin{equation*}
        Q(x)=(x^2_1+\cdots+x^2_n, 2x_1x_2, 2x_1x_3,\dots, 2x_1x_n).
    \end{equation*}
    If $f=(\l_1,\dots,\l_n)$ then
    \begin{equation*}
        f(Q(x))=\l_1x^2_1+\cdots+\l_1x^2_n+2\l_2x_1x_2+
        2\l_3x_1x_3+\cdots+ 2\l_nx_1x_n.
    \end{equation*}
    It is easy to see that
    \begin{eqnarray*}
        &&K'_Q=\{f: \l_1>0, \ \l^2_1>\l^2_2+\l_3^2+\cdots+\l_n^2\},
        \\
        &&{\overline K}'_Q=\{f: \l_1\geq 0,\ \l^2_1\geq
        \l^2_2+\l_3^2+\cdots+\l_n^2\}, \\
        && \bk=\{f\ne 0:
        \l_1=\sqrt{\l^2_2+\l_3^2+\cdots+\l_n^2}\}=\partial {\overline K}'.
    \end{eqnarray*}
    If $f\in \bk$ then $f=(t,\l_2,\l_3,\dots,\l_n),$
    where $t=\sqrt{\l_2^2+\l_3^2+\cdots+\l_n^2},$ and
    \begin{eqnarray*}
        f(Q(x))&=&t(x^2_1+\cdots+x^2_n)+2\l_2x_1x_2+ 2\l_3x_1x_3+\cdots+
        2\l_nx_1x_n\\
        &=&\sum_{i=2}^n\bigg(\frac{\l_i}{\sqrt{t}}+\sqrt{t}x_i\bigg)^2.
    \end{eqnarray*}
It follows from the last equality  that ${\rm rg}f(Q(x))=n-1$ for any
    $f\in \bk.$ Therefor, $Q$ is a homogeneous elliptic quadratic operator of the rank $n-1.$
\end{example}

In the two extreme cases, we may describe the cone ${\overline K}'_Q$ for a homogeneous elliptic quadratic operator.

\begin{enumerate}

\item[(a)] If ${\rm rg}Q=1,$ then $\bk$  consists of exactly $n$ extremal rays and ${\overline K}'_Q$ is a miniedral cone which can be represented as a direct sum of one dimensional cone. The elliptic operator $Q$ can be written by
$Q(x)=(x_1^2,\cdots, x^2_n).$ In this case, the corresponding symmetric bilinear
operator is $B(x,y)=(x_1y_1, x_2y_2,...,x_ny_n)$
and $(\br^n, +, \boxdot)$ is a unital associative, commutative algebra
with multiplication $x\boxdot y =B(x,y)$
and with an identity element $e=(1,1,\dots,1).$

\item[(b)] If ${\rm rg}Q=n-1,$ then $\bk=\partial {\overline K}'_Q,$ i.e. $\overline{K}'_Q$ is a rounded cone in which all boundary rays are extreme.
This means that $\overline{K}'_Q$ can not be represented as a direct sum of two cones. One can check that the following elliptic operator $Q(x)=(x_1^2+\cdots+x^2_n, 2x_1x_2, 2x_1x_3,\cdots, 2x_1x_n)$ has the rank $n-1$ and
the corresponding symmetric bilinear operator has a form
$B(x,y)=(x_1y_1+\cdots+x_ny_n, x_1y_2+x_2y_1, x_1y_3+x_3y_1,\cdots,x_1x_n+x_ny_1).$
If we define a multiplication in $\br^n$ by
$x\odot y =B(x,y)$ then we obtain a commutative algebra $(\br^n,+,\odot)$  with an identity element $e=(1,0,\dots,0)$. It turns out that the multiplicative operation has a property $(x^2\odot y)\odot x=x^2\odot (y\odot x).$ This means that $(\br^n,+,\odot)$ is a Jordan algebra.
\end{enumerate}

In general, it is a tedious work to describe the cone $\overline{K}'_Q$ of homogeneous elliptic quadratic operators with rank $2\leq{\rm rg}Q\leq n-2$.  It can be observed in some examples.

The homogeneous  elliptic quadratic operator having the rank $(n-1)$ has the convexity property.

\begin{definition}[\cite{Sheriff}]
A quadratic operator $Q:\br^n\to \br^m$  is called stably convex if its image $R_{n,m}(Q)$ is convex and it remains convex under sufficiently small
perturbations of $Q$.
\end{definition}

\begin{theorem}[\cite{Sheriff}]
An elliptic quadratic operator $Q:\br^n\to \br^n$ is stably convex if and only if it is homogeneous of the rank  $(n-1)$.
\end{theorem}


\subsection{Some examples: Lower ranks}

In this subsection we are going to consider some concrete examples in the lower dimensional space.

In what follows, we denote by $C_n$ a spherical cone which is an affine similar to the following form
$$C_n=\{f: \l_1\geq 0, \ \l_1^2\geq
\l_2^2+\cdots+\l^2_n\}.$$

\begin{example} Let us consider the following elliptic quadratic operator in $\br^4$
    \begin{equation*}
    Q(x)=(x_1^2+x_2^2, 2x_1x_3, x_2^2+x_3^2, x_2^2+x_4^2).
    \end{equation*}
    Then one can see that
    \begin{equation*}
    K'_Q=\{f: \l_1>0,  \ \l_3>0, \ \l_4>0, \ \l^2_2<\l_1\l_3\},
    \end{equation*}
    and
    \begin{equation*}
    {\overline K}'_Q=C_3\oplus K_1,
    \end{equation*}
    where
    $$
    C_3=\{f: \l_1\geq 0, \  \l_3\geq 0, \ \l_4=0, \
    \l^2_2\leq \l_1\l_3\}
    $$ is a "spherical" cone and $K_1=\{f:
    \l_1=\l_2=\l_3=0; \l_4\geq 0\}$ is a one dimensional cone.
\end{example}

\begin{example} Let us consider the following elliptic quadratic operator in $\br^4$
    \begin{equation*}
    Q(x)=(x_1^2+x_2^2+x^2_3, x_2^2+x_3^2+x_4^2, 2x_1x_3, 2x_2x_4).
    \end{equation*}
    One can check that $Q$ is homogeneous with a rank ${\rm
        rg}Q=2$ and
    $$K'_Q=\{f: \l_1>0, \  \l_2>0, \ \l_3^2<\l_1(\l_1+\l_2), \  \l^2_4<\l_2(\l_1+\l_2)\}.$$
    We denote by
    \begin{eqnarray*}
        && C_3=\{f: \l_1\geq 0, \ \l_2\geq 0, \ \l_3^2\leq \l_1(\l_1+\l_2), \
        \l_4=0\},\\
        && {\widetilde{C}}_3=\{f: \l_1\geq 0, \ \l_2\geq 0, \
        \l_4^2<\l_2(\l_1+\l_2), \  \l_3=0\},\\
        &&R_1=\{f: \l_1=\l_2=\l_3=0\}, \quad {\widetilde R}_1=\{f: \l_1=\l_2=\l_4=0\}.
    \end{eqnarray*}

    Then, we have that
    \begin{equation}\label{137}
    {\overline K}'_Q=(C_3\oplus R_1)\cap ({\widetilde C}_3\oplus {\widetilde R}_1).
    \end{equation}
    $\overline{K}'_Q$ is a intersection of two wedges.
\end{example}

\begin{conjecture} If $\textup{rg}Q=2$ then a cone ${\overline K}'_Q$ is a direct sum of the following cones:
    \begin{enumerate}
        \item[(a)] $C_3$ is a "spherical" cone with dimension three (perhaps, with many copies);

        \item[(b)] $K_r$ is a miniedral cone with dimension $r$ (at most one copy)

        \item[(c)] The cones of type \eqref{137}.
    \end{enumerate}
\end{conjecture}


\subsection{Some examples: Higher ranks}

Let us consider elliptic quadratic operators with $\textup{rg}Q\geq 3$. In this case, it may
appear new type of cones in the decomposition of the cone ${\overline K}'_Q$.

Let $n=\frac{k(k+1)}{2}$ and $x\in \br^n.$ We then write $x$ in a symmetrical matrix form:
$$x=
\left(
\begin{array}{ccccc}
x_1 & x_2 & \cdots & x_k\\[2mm]
x_2 & x_{k+1} & \cdots & x_{2k-1}\\[2mm]
x_3 & x_{k+2} & \cdots & x_{3k-1}\\[2mm]
\cdot & \cdot & \cdots & \cdot\\[2mm]
x_k & x_{2k-1} & \cdots & x_{\frac{k(k+1)}{2}}\\
\end{array}
\right).$$

If we consider $Q(x):=x^2$ then $Q:\br^n\to\br^n$ is a homogenous elliptic quadratic
operator with the rank $k$. In this case, one has that
$$K'_Q=\left\{f=\left(
\begin{array}{cccc}
\l_1 & \cdots & \l_k\\[2mm]
\l_2 & \cdots & \l_{2k-1}\\[2mm]
\cdot & \cdots & \cdot\\[2mm]
\l_k &  \cdots & \l_n\\
\end{array}\right):
f \ {\rm is \ \ positive\  \ defined\ \  matrix}\right\}.
$$

Consequently, ${\overline K}'_Q$ can be identified with the cone
of non negative defined matrices acting on $\br^k.$ We denote this cone by $S_n$, here as before  $n=\frac{k(k+1)}{2}$. One can see that cones $C_3$ and $S_3$ are affine similar. However, for
$k\geq 3$  those cones $C_n$ and $S_n$ could not be affine similar
because of $\textbf{Extr}(C_n)=\partial C_n$ and $\textbf{Extr}(S_n)\ne \partial S_n.$

It was checked in some examples that  there are cones types of $C_4,$ $S_6,$ and $K_r$ in decomposition of the cone ${\overline K}'_Q$ for $\textup{rg}Q=3$ and  there are cones types of $C_5,$ $S_{10}$ and $K_r$ in decomposition of the cone ${\overline K}'_Q$ for $\textup{rg}Q=4.$

However, in general, the description of the cone $\overline{K}'_Q$ for the homogeneous elliptic quadratic operators with a rank of $k$ is a complicated problem. Here we are going to state one problem which is related to the concept of rank.

Let $K$ be a solid cone in the space $\br^n.$ Then due to the Caratheodory  theorem every point $x\in K$ can be presented as a convex hull of at most $n$ extreme vectors of the set $K.$

\begin{definition}
    A number $c(K)$ is called a Caratheodory number of the cone $K$ if it is the smallest number which satisfies the condition: for any $x\in K$ there exists $c\in\bn$ and $x_1,x_2,\cdots,x_c\in{\textbf{\textup{Extr}}}K$ such that $x\in cone\{x_1,x_2,\cdots,x_c\}.$
\end{definition}

It is clear that for any solid cone $K$ we have that $2\leq c(K)\leq n.$ For example, $c(C_n)=2$ and $c(K_n)=n$ where $C_n$ is a spherical cone and $K_n$ is a miniedral cone. By using a spectra theorem for the symmetric matrix one can get that $c\left(S_{\frac{k(k+1)}{2}}\right)=k$ where $S_{\frac{k(k+1)}{2}}$ is given above.

One can easily check that for any two solid cones $K_1$ and $K_2$ we have
$$c(K_1\oplus K_2)\leq c(K_1)+c(K_2).$$

Therefore, if $c(K)\leq 3$ then the cone $K$ could not be represented the direct sum of two cones.

\begin{problem}
    Let $\textup{rg}Q=k$ and $\overline{K}'_Q$ be a closed solid cone corresponding to $Q.$ Find  lower and upper boundaries of the Caratheodory number $c(\overline{K}'_Q)$ of the cone $\overline{K}'_Q.$
\end{problem}



\section{Elliptic quadratic operator equation of rank 1}


Let $Q:\br^n\to\br^n$ is an elliptic quadratic operator of rank 1. Due to Theorem \ref{t9}, we may assume that an elliptic quadratic operator  has the following form
$$
Q(x)=(x_1^2, x_2^2,\dots,x_n^2).
$$
In this case, the elliptic operator equation takes the following form
\begin{equation}\label{138}
x_k^2=\sum^n_{i=1}a_{ki}x_i + b_k; \ k=\overline{1,n}.
\end{equation}
It is obvious that
\begin{eqnarray*}
    &&K'_Q=\{f: \l_1>0,\dots, \l_n>0\},\\
    &&\bk={\rm\textbf{Extr}}({\overline K}'_Q)=\{f_i: \ f_i=(\d_{i1},\dots,\d_{in}),
    \ i=1,\dots,n \},
\end{eqnarray*}
where $\d_{ij}$ is the Kronecker symbol.

We set that
$$
\Pi_i=\{x: x_i^2\leq \sum_{j=1}^na_{ij}x_j+b_i\}, \ i=\overline{1,n}
$$
is a paraboloid corresponding to the functional $f_i\in \bk,$ i.e.
$\Pi_i=\{x: f_i(Q(x)+Ax+b)\leq 0\}.$ According to Theorem \ref{t1}, if the elliptic operator equation \eqref{138} is solvable then the set $\ce_n{(Q,A,b)}$ is nonempty. Therefore, we have that
$\bigcap\limits^n_{i=1}\Pi_i=\ce_n{(Q,A,b)}\neq\emptyset.$

We prove a general fact which is related to any paraboloid of the space $\br^n.$

\begin{lemma}\label{le1}
    Any finite number of paraboloids does not cover $\br^n.$
\end{lemma}

\begin{proof} We know that there exists a hyperplane for every paraboloid such that a paraboloid is symmetric with respect to the hyperplane. We call it \emph{the symmetric hyperplane} of a paraboloid. Let $l$ be a  straight line which is not parallel to any symmetric hyperplane of the given paraboloids. Then the intersection of the straight line $l$  with each symmetric hyperplane is either a finite length segment or an empty set. Since the number of paraboloids is finite the straight line $l$ can not be covered by finite paraboloids.
\end{proof}

\begin{lemma}\label{le2} If $c_i\geq 0, \ i=1,\cdots,n$ then
    \begin{equation*}
    \min_{{\l_i>0\atop
            \sum\l_i=1}}\bigg(\frac{c_1}{\l_1}+...+\frac{c_n}{\l_n}\bigg)=
    (\sqrt{c_1}+...+\sqrt{c_n})^2.
    \end{equation*}
\end{lemma}

This proof is straightforward.

\begin{theorem}\label{twostablesolution}
    Let $A=(a_{ij})_{i,j=1}^n$ be a matrix such that $a_{i_1j}\cdot a_{i_2j}\geq 0, \ \forall i_1,i_2,j=\overline{1,n}.$
    If
    \begin{equation}\label{140}
    \bigg(\sum\limits_{j=1}^n\min\limits_{i=\overline{1,n}}|a_{ij}|\bigg)^2+4 \min\limits_{i=\overline{1,n}} b_i> 0,
    \end{equation}
    then equation \eqref{138} has at least two stable solutions.
\end{theorem}

\begin{remark} In the case $n=1,$ the condition \eqref{140} is nothing but the positivity of the discriminant of the quadratic equation. In this case, it is a necessary and sufficient condition for the existence of two stable solutions.
\end{remark}

\begin{proof} We provide the proof of the theorem in a few steps.

    {\textsf{1 Step.}}

    We will prove that $\ce_f\ne\emptyset$ for any $f\in K'_Q$ as well as  $\ce_n(Q,A,b)\ne \emptyset.$ Indeed, let $f=(\l_1,\dots,\l_n)\in K'_Q$ where $\l_i>0.$ Then, we obtain that
    $$\ce_f=\left\{x: \sum_{i=1}^n\l_i\left(x_i^2-\sum_{j=1}^na_{ij}x_j-b_i\right)\leq
    0\right\},$$ this means
    $$\ce_f=\left\{x:
    \sum_{i=1}^n\l_i\left(x_i-(2\l_i)^{-1}\sum_{j=1}^na_{ji}\l_j\right)^2\leq
    \sum_{i=1}^n\l_i\left(b_i+(2\l_i)^{-2}\left(\sum_{j=1}^na_{ji}\l_j\right)^2\right)\right\}.$$
    Consequently, the set $\ce_f$ is nonempty for any $f\in K'$ if and only
    if
    \begin{equation}\label{141}
    \sum_{i=1}^n\l_i\left(b_i+(2\l_i)^{-2}(\sum_{j=1}^na_{ji}\l_j)^2\right)\geq
    0,
    \end{equation}
    for any $\l_1>0,\cdots,\l_n>0.$ Since $\ce_{\a f}=\ce_f$ for
    $\a>0$, it is sufficient to check \eqref{141} for
    $\l_1>0,\cdots,\l_n>0$ with $\sum\limits^n_{i=1}\l_i=1.$ We denote by
    $$
    c_i=\frac{1}{4}\left(\sum\limits^n_{j=1}a_{ji}\l_j\right)^2.
    $$
    According to Lemma \ref{le2}, we get that
    \begin{equation*}
    \sum_{i=1}^n\frac{\left(\sum\limits^n_{j=1}a_{ji}\l_j\right)^2}{4\l_i}\geq
    \frac{1}{4}\left(\sum_{i=1}^n\left|\sum_{j=1}^na_{ji}\l_j\right|\right)^2.
    \end{equation*}
    Since the signs of elements of each column of a matrix $A$ are the same and $\sum\limits_{i=1}^n\l_i=1,$ one can obtain
    \begin{equation}\label{142}
    \frac{1}{4}\left(\sum_{i=1}^n\left|\sum_{j=1}^na_{ji}\l_j\right|\right)^2\geq
    \frac{1}{4}\left(\sum_{i=1}^n\max_j|a_{ji}|\right)^2
    \end{equation}
    and
    \begin{equation}\label{143}
    \sum_{i=1}^n\l_i b_i\geq \min_i b_i.
    \end{equation}
    It follows from \eqref{142} and \eqref{143} that
    \begin{equation*}
    \frac{1}{4}\sum_{i=1}^n\frac{\left(\sum_{j=1}^na_{ji}\l_j\right)^2}{\l_i}+\sum_{i=1}^n\l_i
    b_i\geq
    \frac{1}{4}\left[\left(\sum_{i=1}^n\min_j|a_{ji}|\right)^2+4\min_i
    b_i\right]>0.
    \end{equation*}
Hence,  we have that $\ce_f\ne\emptyset$ for any
    $f\in K'$ as well as $\ce_n(Q,A,b)\ne \emptyset$ (see Remark \ref{efimplye}).

    {\textsf{2 Step.}}

    We will prove that $\ce_n(Q,A,b)$ is a solid set. To do that we use the following representation of the set $\ce_n(Q,A,b)$
    $$\ce_n(Q,A,b)=\bigcap_{i=1}^n \Pi_i,$$
    where $\Pi_i=\{x: x_i^2\leq a_{i1}x_1+...+a_{in}x_n+b_i\}$ is a
    convex solid set. We assume the contrary i.e. ${\textup{int}}\ce_n(Q,A,b)=\emptyset.$ Let $k$ be a
    number such that an intersection of arbitrary $k$ sets from the
    family of sets $\{\Pi_1,\cdots,\Pi_n\}$ is solid, but there are $k+1$ sets
    from that family such that an intersection of which is not solid.
    Without loss of generality, we may assume that they are $\Pi_1,
    \Pi_2,\dots,\Pi_{k+1},$ with $\bigcap\limits_{i=1}^{k+1}\Pi_i=\emptyset.$ We put $\D_k=\bigcap\limits_{i=1}^k\Pi_i.$ Then, we get that
    \begin{equation}\label{144}
    {\rm int}\bigcap\limits_{i=1}^{k+1}\Pi_i={\rm int}\D_k\cap {\rm int}\Pi_{k+1}=\emptyset.
    \end{equation}
    Since both $\D_k$ and $\Pi_{k+1}$ are solid sets, it follows from \eqref{144}
    that
    \begin{equation}\label{145}
    \D_k\cap {\rm int}\Pi_{k+1}=\emptyset.
    \end{equation}
    Let $H=\{x\in\mathbb{R}^n: \varphi(x)=c\}$ be a hyperplane separating $\D_k$ and $\Pi_{k+1}$. Without loss of generality, we may assume  that $\varphi(x)\geq c$
    for any $x\in \Pi_{k+1}.$ We setup
    $$\Pi_{k+1}^{(\e)}=\{x: x^2_{k+1}\leq
    a_{k+1,1}x_1+\cdots+a_{k+1,n}x_n+b_{k+1}-\e\}.$$ We then
    have that $\pi^{(\e)}_{k+1}\subset {\rm int}\pi_{k+1}$ for any $\e>0.$  Due to inequality \eqref{145},
    one can find that
    $$\D_k\cap \Pi^{(\e)}_{k+1}=\emptyset.$$
    Thus, for a perturbed  elliptic operator equation
    $$\left\{\begin{array}{lllllll}
    x_1^2=a_{11}x_1+\cdots+a_{1n}x_n+b_1\\
    \dots \dots\dots \\
    x_k^2=a_{k1}x_1+\cdots+a_{kn}x_n+b_k\\
    x_{k+1}^2=a_{k+1,1}x_1+\cdots+a_{k+1,n}x_n+b_{k+1}-\e\\
    x_{k+2}^2=a_{k+2,1}x_1+\cdots+a_{k+2,n}x_n+b_{k+2}\\
    \dots \dots\dots\\
    x_n^2=a_{n1}x_1+\cdots+a_{nn}x_n+b_n,\\
    \end{array}\right.
    $$
    we have that
    $$\D_{\e}=\Pi_1\cap \cdots \cap\Pi_k\cap \Pi^{(\e)}_{k+1}\cap\Pi_{k+2}\cap \cdots
    \cap \Pi_n=\emptyset,$$ for any $\e>0.$

    On the other hand, it is clear that the condition
    \eqref{140} is satisfied for a sufficient small $\e,$ i.e.,
    \begin{equation*}
    \bigg(\sum\limits_{j=1}^n\min\limits_{i=\overline{1,n}}|a_{ij}|\bigg)^2+4\min\limits_{i=\overline{1,n}}b_i-\e>0,
    \end{equation*}
    for a sufficiently small $\e>0.$ If we apply the first step to $\D_{\e}$ then we have that $\D_{\e}\ne \emptyset.$ However, this is a contradiction. Hence, $\ce_n(Q,A,b)$ is a solid set.

    {\textsf{3 Step.}}

    For the sake of simplicity we will assume that $a_{ij}\geq 0.$
    In general, this condition can be achieved  by a transformation:
    $$(x_1,x_2,\dots,x_n)\to (\pm x_1, \pm x_2, \dots, \pm x_n),$$
    where we take $"+"$  in front of $x_k$ if $k-$th column of the
    matrix $A$ is not negative and  otherwise $"-"$.

    {\textsf{4 Step.}}

    For $x,y\in \ce_n(Q,A,b),$ we will prove that $z=x\vee y\in \ce_n(Q,A,b),$ here as
    before
    $$z=x\vee y=(\max(x_1,y_1), \max(x_2,y_2),\dots,\max(x_n,y_n)).$$
    Indeed, if $x,y\in \ce_n(Q,A,b)$ then
    \begin{eqnarray*}
        && x_k^2\leq \sum_{i=1}^na_{ki}x_i+b_k, \ k=1,2,\cdots,n,\\
        && y_k^2\leq \sum_{i=1}^na_{ki}y_i+b_k, \ k=1,2,\cdots,n.
    \end{eqnarray*}
    Since $a_{ki}\geq 0$ we have
    \begin{eqnarray*}
        z_k^2&=&(\max(x_k,y_k))^2\leq \max(\sum_ia_{ki}x_i+b_k,
        \sum_ia_{ki}y_i+b_k)\\
        &\leq&\max(\sum_ia_{ki}x_i, \sum_ia_{ki}y_i)+b_k \leq\sum_ia_{ki}\max(x_i,y_i)+b_k=\sum_ia_{ki}z_i+b_k.
    \end{eqnarray*}
This means that $z\in \ce_n(Q,A,b).$

    {\textsf{5 Step.}}

    It follows from the previous steps that $\ce_n(Q,A,b)$ is a closed  solid  convex and bounded set which is closed under the operation $x\vee y$  whenever $x,y\in \ce_n(Q,A,b).$ We setup
    $$
    x^*=\sup\ce_n(Q,A,b)=\left(\max_{x\in \ce_n(Q,A,b)}\{x_1\}, \cdots,\max_{x\in \ce_n(Q,A,b)}\{x_n\}\right).
    $$
    It is obvious that $x^*$ is a vertex of $\ce_n(Q,A,b).$ According to Theorem \ref{t6}, it is a
    stable solution of the elliptic operator equation \eqref{138}. The existence of the
    second stable solution follows from Theorem \ref{t7}. This completes the proof.
\end{proof}

\begin{remark} In general, $\ce_n(Q,A,b)$ is not closed with respect to the operation
    $\wedge$, i.e. from $x, y\in \ce_n(Q,A,b)$ it does not follow that $x\wedge
    y\in\ce_n(Q,A,b)$, where
    $$x\wedge y=(\min(x_1,y_1),\dots, \min(x_n,y_n)).$$
\end{remark}

In conclusion of this section, we will give a constructive method of finding
the stable solution $x^*$ which was shown in Theorem \ref{twostablesolution}.

We assume that the condition \eqref{140} is satisfied and
$a_{ij}\geq 0$ for any $i,j=\overline{1,n}.$ We setup
$$M=\max_i\sum^n_{j=1}a_{ij}, \ b=\max_i b_i,$$
and we choose $\a>0$ such that
\begin{equation}\label{146}
\a> \frac{1}{2}(M+\sqrt{M^2+4b}).
\end{equation}

\begin{theorem}\label{t13} Let the condition \eqref{140} and $a_{ij}\geq 0$ for any $i,j=\overline{1,n}$
    be satisfied. Then the Newton-Kantorovich iteration with an initial
    point $x_0=(\a,\a,\dots,\a)$ converges to a solution $x^*.$
\end{theorem}

\begin{proof} It is enough to check the conditions of Theorem \ref{t8}. We prove it in a few steps.

    {\textsf{1 Step.}}

    Due to Lemma \ref{le1}, we have that
    $$\bd=\br^n\setminus \left(\bigcup^n_{i=1}\Pi_i\right)\ne \emptyset.$$
    One can see  that $\bd$ is an open set. Let us check that $x_0=(\a,\a,\dots,\a)\in
    \bd.$ It follows from \eqref{146} that $\a^2>M\a+b$ and
    $$\a^2>M\a+b\geq \sum^n_{j=1}a_{ij}\a+b_i, \ \forall i=\overline{1,n},$$
    i.e. $x_0\notin \Pi_i$ for any $i=\overline{1,n}.$ Hence, we obtain that $x_0\in \bd.$

    Let $\bd_0$ be a connected component of $\bd$ with $x_0\in \bd_0.$

    {\textsf{2 Step.}}

    We will show that there is no a straight line passing through $x_0$
    and containing inside ${\overline \bd}_0.$ We assume the contrary, i.e.,
    $\{x_0+\l h\}_{\lambda\in\mathbb{R}}\subset {\overline \bd}_0,$ $h\ne 0.$ Then one can get that
    $$(\a+\l h_k)^2\geq \sum_ia_{ki}(\a+\l h_i)+b_k,  \ \ i=1,\dots,n$$
    i.e.
    \begin{equation}\label{147}
    \bigg(2\a h_k-\sum_ia_{ki} h_i\bigg)^2-4h^2_k\bigg(\a^2-\a
    \sum_ia_{ki}-b_k\bigg)\leq 0, \ \  i=1,\dots,n.
    \end{equation}
    We denote by
    $$m=\sum_j\min_i a_{ij}, \ \ \b=\min_i b_i.$$
    One can choose $k_0$ such that $h_{k_0}\ne 0$
    and
    \begin{equation}\label{148}
    \left|2\a h_{k_0}-\sum_ia_{k_0i} h_i\right|\geq \left|2h^2_{k_0}-m
    h_{k_0}\right|.
    \end{equation}
    We then have that
    \begin{eqnarray*}
        \bigg(2\a h_{k_0}-\sum_ia_{k_0i}
        h_i\bigg)^2-4h^2_{k_0}\bigg(\a^2-\a
        \sum_ia_{k_0i}-b_{k_0}\bigg)&\geq& h_{k_0}^2\bigg(m^2-4\a m+4\a \sum_ia_{k_0i}+4b_{k_0}\bigg)\\
        &\geq& h_{k_0}^2(m^2+4\b).
    \end{eqnarray*}

    Due to condition \eqref{140}, we find $m^2+4\b>0.$ Hence,
    $$
    \bigg(2\a h_{k_0}-\sum_ia_{k_0i}
    h_i\bigg)^2-4h^2_{k_0}\bigg(\a^2-\a
    \sum_ia_{k_0i}-b_{k_0}\bigg)>0, $$ which contradicts to the inequalities \eqref{147}.

    Thus, there is no a straight line passing through $x_0$ and containing in ${\overline \bd}_0.$ According to Theorem \ref{t8} the Newton-Kantorovich iteration $$x^{(m+1)}=x^{(m)}-[P'(x^{(m)})]^{-1}P(x^{(m)}), \  m=0,1,\dots$$
    exists and lies in $\bd_0$. Moreover, it converges to a stable
    solution of the elliptic operator equation.

    {\textsf{3 Step.}}

    Now, we will prove that the Newton-Kantorovich iteration  $\{x_m\}$ converges to $x^*$ which was found in the proof of Theorem \ref{twostablesolution}. Indeed, it follows from $x\in \bd_0$  that $x\geq x^*$ i.e. $x_i\geq x_i^*$ for all $i=\overline{1,n}.$ In particular, since $\{x^{(m)}\}\subset \bd_0,$ we have
    $x^{(m)}\geq x^*$ for any $m\in\bn.$ If $x^{(m)}\to {\bar{x}},$ then ${\bar{x}}\geq x^*.$
    The relations  ${\bar{x}}\in\ce_n(Q,A,b)$ and $x^*=\sup\limits_{x\in \ce_n(Q,A,b)}\{x\}$ yield
    that ${\bar{x}}=x^*.$ This completes the proof.
\end{proof}

%
%
%
%
%
%
%
%

\end{document}